\documentclass[a4paper, 12pt]{article}
\usepackage{geometry}
\usepackage{amsmath, amsthm, amsfonts, amssymb}

\allowdisplaybreaks

\usepackage[monochrome]{xcolor}

\usepackage{enumerate, graphicx, subfig, ulem}
\usepackage{cancel}
\renewcommand{\uwave}{}

\usepackage{tikz}\usepackage{pgfplots}\usepackage{array, float}

\usepackage{titling}
\thanksmarkseries{alph}

\usepackage[driverfallback=dvipdfm]{hyperref}
\usepackage[nameinlink]{cleveref}
\hypersetup{colorlinks=true, linkcolor=teal, citecolor=blue, urlcolor=blue}

\newcommand{\real}{\mathbb{R}}

\DeclareMathOperator{\tr}{tr}
\DeclareMathOperator{\End}{End}\DeclareMathOperator{\Sym}{Sym}

\DeclareMathOperator{\Vol}{Vol}

\DeclareMathOperator{\grad}{grad}\DeclareMathOperator{\Ric}{Ric}\DeclareMathOperator{\Hess}{Hess}
\DeclareMathOperator{\Met}{Met}

\theoremstyle{plain}
\newtheorem{proposition}{Proposition}[section]

\newtheorem{lemma}[proposition]{Lemma}
\newtheorem{theorem}[proposition]{Theorem}
\newtheorem*{theorem*}{Theorem}
\theoremstyle{definition}
\newtheorem{definition}[proposition]{Definition}
\newtheorem{remark}[proposition]{Remark}

\numberwithin{equation}{section}

\usepackage{etoolbox}
\makeatletter
\newcounter{author}
\renewcommand*\author[1]{%
  \stepcounter{author}%
  \ifnum\c@author=1
    \gdef\@author{#1}%
  \else
    \xdef\@author{\unexpanded\expandafter{\@author\and#1}}%
  \fi
  \csgdef{author@\the\c@author}{#1}}
\newcommand*\email[1]{%
  \csgdef{email@\the\c@author}{#1}}
\newcommand*\address[1]{%
  \csgdef{address@\the\c@author}{#1}}
\AtEndDocument{%
  \xdef\author@count{\the\c@author}%
  \c@author=1
  \print@authors}
\newcommand*\print@authors{%
  \ifnum\c@author>\author@count
  \else
    \print@author{\the\c@author}%
    \advance\c@author by 1
    \expandafter\print@authors
  \fi}
\newcommand*\print@author[1]{%
  \par\medskip
  \begin{tabular}{@{}l@{}}%
    \textsc{\csuse{author@#1}}\\
    \csuse{address@#1}\\
    \textit{E-mail address}:
    \href{mailto:\csuse{email@#1}}{\csuse{email@#1}}
  \end{tabular}}
\makeatother

\usepackage{sectsty}
\sectionfont{\normalsize}
\subsectionfont{\normalsize}
\subsubsectionfont{\normalsize}

\title{Scalar curvature and the multiconformal class of a direct product Riemannian manifold}
\author{Nobuhiko Otoba}
\address{Universit\"at Regensburg, 93040 Regensburg, Germany}
\email{nobuhiko.otoba@ur.de}
\author{Saskia Roos}
\address{Max Planck Institute for Mathematics, Vivatsgasse 7, 53111 Bonn, Germany}
\email{saroos@mpim-bonn.mpg.de}

\date{}

\begin{document}
\maketitle

\begingroup
\renewcommand{\thefootnote}{}
\endgroup

\begin{abstract}
For a closed, connected direct product Riemannian manifold $(M, g)=(M_1\times\cdots\times M_l, g_1\oplus\cdots\oplus  g_l)$, we define its multiconformal class $ [\![ g ]\!]$ as the totality $\{f_1^2g_1\oplus \cdots\oplus f_l^2g_l\}$ of all Riemannian metrics obtained from multiplying the metric $g_i$ of each factor $M_i$ by a function $f_i^2>0$ on the total space $M$.  A multiconformal class $ [\![ g ]\!]$ contains not only all warped product type deformations of $g$ but also the whole conformal class $[\tilde{g}]$ of every $\tilde{g}\in [\![ g ]\!]$. In this article, we prove that $ [\![ g ]\!]$ contains a metric of positive scalar curvature if and only if the conformal class of some factor $(M_i, g_i)$ does, under the technical assumption $\dim M_i\ge 2$. We also show that, even in the case where every factor $(M_i, g_i)$ has positive scalar curvature, $ [\![ g ]\!]$ contains a metric of scalar curvature constantly equal to $-1$ and with arbitrarily large volume, provided $l\ge 2$ and $\dim M\ge 3$. In this case, such negative scalar curvature metrics within $ [\![ g ]\!]$ for $l=2$ cannot be of any warped product type. 
\end{abstract}

\tableofcontents

\section{Introduction}
The notion of \textit{warped products} dates back to the work of Bishop--O'Neill \cite{MR0251664}, in which they showed that a variety of product manifolds carry metrics of negative sectional curvature. 
Here, for two Riemannian manifolds $(B, \check{g})$ and $(F, \hat{g})$, their warped product with respect to a positive function $f$ on the base space $B$ is defined as the Riemannian manifold $(B\times F, \check{g}\oplus f^2\hat{g})$. 
%
Another ubiquitous way of deforming a Riemannian metric by functions is \textit{conformal change}. 
Here, for a Riemannian manifold $(M, g)$, we take a positive function $f$ on $M$ and define a new Riemannian manifold $(M, f^2g)$, in which the angle of two tangent vectors is the same as that of $(M, g)$. 
%
Unifying the notions of warped product and conformal change, Koike \cite{MR1084347} defined the \textit{twisted product} of two Riemannian manifolds $(M_1, g_1)$, $(M_2, g_2)$ with respect to positive functions $f_1, f_2$ on $M_1\times M_2$ as the Riemannian manifold $(M_1\times M_2, f_1^2g_1\oplus f_2^2g_2)$. 
See also Koike \cite{MR1212293} for the case of more than two factors; 
for more details about the terminology, we refer to Remark \ref{rmk:terminology}. 
His main interest seemed to be the extrinsic geometry of the leaves $M_1$ and $M_2$, such as total umbilicity. 
It is the aim of the present article to study twisted product metrics in the sense of Koike from a more intrinsic point of view, with particular focus on scalar curvature. 

Here and henceforth, we adopt the following notation. 
For a smooth manifold $M$, we define $C^{\infty}(M)=\{\varphi: M\to\real\mid \text{$\varphi$ is $C^{\infty}$ smooth}\}$, $C^{\infty}_+(M)=\{u\in C^{\infty}(M)\mid u>0\}$. 
For a Riemannian metric $g$ on $M$, $R^g$ and $d\mu^g$ denote the scalar curvature and the volume element of $g$, respectively. 
For $\varphi\in C^{\infty}(M)$, $\grad^g\varphi$ and $\Delta^g\varphi$ are the gradient vector field and the Laplacian of $\varphi$, respectively, so that $\int_M\lvert\grad^g\varphi\rvert^2d\mu^g=-\int_M\varphi(\Delta^g\varphi) d\mu^g$ if $M$ is closed (i.e.~compact without manifold boundary). 

Firstly, we recall some fundamental results about the scalar curvature of Riemannian metrics within a conformal class. 
Let $(M^m, g)$ be a closed connected Riemannian manifold. The scalar curvature of the conformally related metric $\tilde{g}=f^2g$, $f\in C^{\infty}_+(M)$, satisfies
\begin{align}\label{eq:PDEConformalChange}
-\frac{4(m-1)}{m-2}\Delta^gu+R^gu=R^{\tilde{g}}u^{\frac{m+2}{m-2}}
\end{align}
if we assume $m\ge 3$ and set $u=f^{(m-2)/2}$. 
Let
\begin{align*}
\lambda_0=\inf_{\substack{\varphi: M\to\real, \\ \varphi\not\equiv 0}}\frac{\int_{M}\left(\frac{4(m-1)}{m-2}\lvert\grad^{g}\varphi\rvert^2+R^{g}\varphi^2\right)d\mu^{g}}{\int_{M}\varphi^2d\mu^{g}}>-\infty
\end{align*}
be the smallest eigenvalue of the operator $-\frac{4(m-1)}{m-2}\Delta^g+R^g$ and $v$ the corresponding eigenfunction so normalized that $\min_{M}v=1$. 
On the one hand, setting $u=v$ in \eqref{eq:PDEConformalChange}, we see that the metric $v^{4/(m-2)}g$ has scalar curvature $\lambda_0v^{-4/(m-2)}$, which has the same sign as $\lambda_0$. 
On the other hand, $\Delta^g u$ attains zero somewhere on $M$ because $\int_M(\Delta^gu)d\mu^g=0$, 
and thus $R^g$ and $R^{\tilde{g}}$ cannot have different signs by \eqref{eq:PDEConformalChange}. 
It follows that, provided $m\ge 3$, every conformal class $[g]$ on a closed connected manifold contains a metric $\tilde{g}$ whose scalar curvature satisfies either $R^{\tilde{g}}>0$, $R^{\tilde{g}}\equiv 0$, or $R^{\tilde{g}}< 0$, and these three cases are mutually exclusive. 
The same statement holds also for $m=1$ and $2$, where we note that the scalar curvature of a $1$-dimensional Riemannian manifold is always constantly equal to zero and that $\int_{M^2}R^gd\mu^g=4\pi\chi(M^2)$ in dimension $2$ by the Gauss--Bonnet theorem. 
It is now well known that there exists a metric of constant scalar curvature in every conformal class of a closed connected manifold $M^m$; 
this follows from the uniformization theorem for Riemann surfaces if $m=2$ and from the resolution of the Yamabe problem if $m\ge 3$ (cf.~Yamabe \cite{MR0125546}, Trudinger \cite{MR0240748}, Aubin \cite{MR0431287}, Schoen \cite{MR788292}). 

Secondly, we recall the following related observations on the scalar curvature of warped product metrics, which can be found in Dobarro--Lami Dozo \cite[Theorems 3.1--3.3]{MR896013}. Let $(M^m, g)=(M_1^{m_1}, g_1)\times (M_2^{m_2}, g_2)$ be a direct product of closed connected Riemannian manifolds, $f_2\in C^{\infty}_+(M_1)$, and $\tilde{g}=g_1\oplus f_2^2g_2$ a warped product metric on $M$. Then 
\begin{align}\label{eq:DobarroLamiDozoEq}
-\frac{4m_2}{m_2+1}\Delta^{g_1}u+R^{g_1}u+R^{g_2}u^{\frac{m_2-3}{m_2+1}}=R^{\tilde{g}}u
\end{align}
where we set $u=f_2^{(m_2+1)/2}$. Differentiation of \eqref{eq:DobarroLamiDozoEq} with respect to vector fields tangent to $M_2$ shows that $R^{g_2}$ is constant if $R^{\tilde{g}}$ is constant. 
We define $\lambda_0\in\real$ and $v\in C^{\infty}_+(M_1)$ so that  
\begin{align*}
&\lambda_0=\inf_{\substack{\varphi: M_1\to\real, \\ \varphi\not\equiv 0}}\frac{\int_{M_1}\left(\frac{4m_2}{m_2+1}\lvert\grad^{g_1}\varphi\rvert^2+R^{g_1}\varphi^2\right)d\mu^{g_1}}{\int_{M_1}\varphi^2d\mu^{g_1}}, \\
&-\frac{4m_2}{m_2+1}\Delta^{g_1}v+R^{g_1}v=\lambda_0 v, \quad \min_{M_1} v=1. 
\end{align*}
Integration by parts then yields 
\begin{align*}
(R^{\tilde{g}}-\lambda_0)\int_{M_1}uvd\mu^{g_1}
=R^{g_2}\int_{M_1}u^{\frac{m_2-3}{m_2+1}}vd\mu^{g_1}, 
\end{align*}
provided both $R^{\tilde{g}}$ and $R^{g_2}$ are constant; 
in particular, $R^{\tilde{g}}-\lambda_0$ and $R^{g_2}$ must have the same sign. 
We note that $\lambda_0>0$, $\lambda_0=0$, and $\lambda_0<0$ hold, respectively, if $R^{g_1}\ge 0$ and $R^{g_1}\not\equiv 0$, if $R^{g_1}\equiv 0$, and if $\int_{M_1}R^{g_1}d\mu^{g_1}\le 0$ and $R^{g_1}\not\equiv 0$. 
We summarize the sign restrictions thus obtained for warped product metrics of constant scalar curvature in Table \ref{table:WapredProducts}. 
\begin{table}[h]\centering
\begin{tabular}{c|ccc}
&$R^{g_2}>0$&$R^{g_2}=0$&$R^{g_2}<0$\\\hline
$R^{g_1}\ge 0$, $R^{g_1}\not\equiv 0$&$R^{\tilde{g}}>0$&$R^{\tilde{g}}>0$&\\
$R^{g_1}\equiv 0$&$R^{\tilde{g}}>0$&$R^{\tilde{g}}=0$&$R^{\tilde{g}}<0$\\
$\int_{M_1}R^{g_1}d\mu^{g_1}\le 0$, $R^{g_1}\not\equiv 0$&&$R^{\tilde{g}}<0$&$R^{\tilde{g}}<0$
\end{tabular}
\caption{Warped product metrics of constant scalar curvature}\label{table:WapredProducts}
\end{table}
We take the direct product of constant scalar curvature metrics to see that each case in Table \ref{table:WapredProducts} is nonempty and that, rescaling the metric $g_2$ by constants, there is no sign restriction on the scalar curvature in the two remaining cases in Table \ref{table:WapredProducts} that are left blank. 
The existence question for warped product metrics of constant scalar curvature which are not direct product is also considered in the same article \cite{MR896013}. 
The scalar curvature of warped product type metrics are also studied in \cite{MR563189, MR695711, MR1406047, MR1651555, MR1639860, MR1822102, MR2157416, MR2442173, MR2434666}. 

We introduce the following perspective that unifies the previous sign restrictions on scalar curvature. 
For a direct product Riemannian manifold $(M, g)=(M_1\times\cdots\times M_l, g_1\oplus \cdots\oplus  g_l)$, we define its multiconformal class $ [\![ g ]\!]$ by 
\begin{equation*}
 [\![ g ]\!]=\{f_1^2g_1\oplus \cdots\oplus f_l^2g_l\mid f_1, \dots, f_l\in C^{\infty}_+(M)\}. 
\end{equation*}
We emphasize that the multiconformal factor $f_i$ does not have to be constant along any $M_j$. 
Within a multiconformal class, we may 
(1) conformally change the representative metrics $g_1, \dots, g_l$ on $M_1, \dots, M_l$, respectively, 
(2) conformally deform an arbitrary metric $\tilde{g}\in [\![ g ]\!]$, and 
(3) consider warped product type metrics of all kinds.  
Our main result is the following trichotomy. See Remark \ref{rmk:DimensionalAssumption} regarding its dimensional assumption. 
\begin{theorem}\label{thm:main}
Let $(M^m, g)=(M_1^{m_1}\times\cdots\times M_l^{m_l}, g_1\oplus \cdots\oplus  g_l)$ be a direct product of closed connected Riemannian manifolds. 
Assume $m\ge 3$, $l\ge 2$, and $m_1, \dots, m_l\ge 2$. Then the following trichotomy holds. 
\begin{enumerate}
	\item The multiconformal class $ [\![ g ]\!]$ contains a metric of positive scalar curvature 
	if and only if there exists $i\in\{1, \dots, l\}$ such that the conformal class $[g_i]$ contains a metric of positive scalar curvature. 
	\item The multiconformal class $ [\![ g ]\!]$ does not contain a metric of positive scalar curvature 
	and there exists a scalar flat metric of $ [\![ g ]\!]$ 
	if and only if $[g_i]$ contains a scalar flat metric for every $i\in\{1, \dots, l\}$. 
	In this case, if $\tilde{g}\in [\![ g ]\!]$ has nonnegative scalar curvature, 
	then $\tilde{g}$ is necessarily scalar flat and direct product. 
	\item The multiconformal class $ [\![ g ]\!]$ does not contain a metric of nonnegative scalar curvature 
	if and only if $[g_i]$ does not contain a metric of nonnegative scalar curvature for every $i\in\{1, \dots, l\}$ 
	and there exists $i\in\{1, \dots, l\}$ such that $[g_i]$ contains a metric of negative scalar curvature.  
\end{enumerate}
\end{theorem}

Our next result, Theorem \ref{thm:sub1}, shows that even in the case (1) of Theorem \ref{thm:main}, $ [\![ g ]\!]$ contains a metric of scalar curvature constantly equal to $-1$ and with arbitrarily large volume, provided $l\ge 2$ and $\dim M\ge 3$. We observe from Theorem \ref{thm:sub2} that  such negative scalar curvature metrics within $ [\![ g ]\!]$ for $l=2$ cannot be of any generalized warped product type (cf.~Table \ref{table:WapredProducts}). 

\begin{theorem}\label{thm:sub1}
There exists a sequence $\{\tilde{g}^{(n)}\}_{n=1}^{\infty}\subset [\![ g ]\!]$ of metrics multiconformal to $g$ such that the scalar curvature of $\tilde{g}^{(n)}$ is constantly equal to $-1$ on $M$ for every $n\ge 1$ and that $\Vol\left(M, \tilde{g}^{(n)}\right)\to\infty$ as $n\to\infty$. 
\end{theorem}

\begin{theorem}\label{thm:sub2}
Assume $l=2$, $R^{g_1}, R^{g_2}\ge 0$. If $f_i$ is constant along $M_i$ for $i=1, 2$, then the scalar curvature of $f_1^2g_1\oplus  f_2^2g_2$ cannot be nonpositive everywhere and negative somewhere at the same time. 
\end{theorem}

This article is organized as follows. 
In Sect.~\ref{sect:EinsteinHilbertRestrictions}, 
	we characterize criticality with respect to the normalized Einstein--Hilbert functional restricted to a multiconformal class. 
In Sect.~\ref{sect:KillFirst}, 
	we introduce some differential operators and observe their behavior under a change of dependent variables.  
In Sect.~\ref{sect:Karcher}, 
	we compute the scalar curvature of a multiconformally related metric. 
In Sect.~\ref{sect:Formulas}, 
	we derive integral formulas which play a crucial role in the proof of Theorem \ref{thm:main}. 
We prove the trichotomy theorem (Theorem \ref{thm:SignOfMulticonformalClass}) in Sect.~\ref{sect:MulticonformalSign}, which is equivalent to Theorem \ref{thm:main}.  
The proofs of Theorems \ref{thm:sub1}, \ref{thm:sub2} are in Sects.~\ref{sect:InfYamabeConst} and \ref{sect:WarpedProduct}, respectively. 

\section{The normalized Einstein--Hilbert functional}\label{sect:EinsteinHilbertRestrictions}
Let $E: \Met\to\real$ be the normalized Einstein--Hilbert functional defined on the space $\Met$ of all Riemannian metrics on a closed connected manifold $M^m$ by 
\begin{equation*}
E(g)=\frac{\int_MR^gd\mu^g}{(\int_Md\mu^g)^{2/p_m}}. 
\end{equation*}
Here, $R^g$ is the scalar curvature of $g$, $p_m=2m/(m-2)$, and $m\ge 3$. 
For every $h\in\Gamma(\Sym^2TM^*)$, 
\begin{equation*}
\left.\frac{d}{dt}\right|_{t=0}E(g+th)
=\int_M\langle h, (2^{-1}R^g-p_m^{-1}r^g)g-\Ric^g\rangle d\mu^g
\end{equation*}
where $r^g=\int_MR^gd\mu^g/\int_Md\mu^g$.  
It follows that critical points of $E$ itself and $E$ restricted to a conformal class $[g]$, respectively, are precisely Einstein metrics on $M$ and metrics of constant scalar curvature within $[g]$. 
Note that the first variation formula for $E$ simplifies to 
\begin{equation}\label{eq:FirstVariationFormulaEinsteinHilbert}
\left.\frac{d}{dt}\right|_{t=0}E(g+th)
=-\int_M\langle h, \Ric^g-(R^g/m)g\rangle d\mu^g
\end{equation}
at a metric $g$ of constant scalar curvature. 

We introduce two intermediate notions of criticality between the constant scalar curvature and Einstein conditions. 
Fix a direct sum decomposition $TM=E_1+\cdots+E_l$ of the tangent bundle 
where each $E_i$ is a vector subbundle of fiber dimension $m_i\ge 1$, 
and let $\mathcal{P}_i\in\Gamma(\End TM)$ be the corresponding projection onto $E_i$. 
We say $g\in\Met$ is compatible with this decomposition if $E_i\perp E_j$ for all $i\neq j$ with respect to $g$. 
We denote by $\Met_{\perp}$ the subspace of all compatible Riemannian metrics. 
Note that $g\in\Met_{\perp}$ can be written uniquely as $g=g_1\oplus \cdots\oplus  g_l$ using fiberwise inner products $g_1, \dots, g_l$ on $E_1, \dots, E_l$, respectively.  
Two compatible metrics $\tilde{g}=\tilde{g}_1\oplus \cdots\oplus \tilde{g}_l$ and $g=g_1\oplus \cdots\oplus  g_l$ are said to be multiconformal to each other if there exist functions $f_1, \dots, f_l: M\to\real_{>0}$ such that $\tilde{g}_i=f_i^2g_i$ for all $i\in\{1, \dots, l\}$. 
Multiconformality defines an equivalence relation on $\Met_{\perp}$. 
We denote by $ [\![ g ]\!]$ the equivalence class of $g\in\Met_{\perp}$. 
For $g\in\Met_{\perp}$, we have 
\begin{equation}\label{eq:FundamentalInclusion}
[g]\subset [\![ g ]\!]\subset\Met_{\perp}\subset\Met. 
\end{equation}

For a compatible metric $g$, define  $\Ric^g_i\in\Gamma(\Sym^2TM^*)$ and $R^g_i\in C^{\infty}(M)$ by 
\begin{align*}
\Ric^g_i(X, Y)=\Ric^g(\mathcal{P}_iX, \mathcal{P}_iY), 
&&
R^g_i=\langle \Ric^g, g_i\rangle=\tr^g\Ric^g_i
\end{align*}
for all $X, Y\in\Gamma(TM)$. 
We note that 
\begin{equation}\label{eq:RIsSimplyTheSum}
R^g=R^g_1+\cdots+R^g_l
\end{equation}
always holds while $\Ric^g=\Ric^g_1+\cdots+\Ric^g_l$ holds if and only if $\Ric^g(\mathcal{P}_iX, \mathcal{P}_jY)=0$ for all $X, Y\in\Gamma(TM)$ whenever $i\neq j$. 

\begin{proposition}
Let $g=g_1\oplus \cdots\oplus  g_l\in\Met_{\perp}$ be a compatible metric. 
\begin{enumerate}
\item $g$ is critical with respect to the functional $E$ restricted to $ [\![ g ]\!]$ if and only if there exists a real number $c$ independent of $i$ such that $R^g_i/m_i=c$ for all $i\in\{1, \dots, l\}$. 
\item $g$ is critical with respect to the functional $E$ restricted to $\Met_{\perp}$
if and only if there exists a constant $c$ independent of $i$ such that $\Ric^g_i=cg_i$ for all $i\in\{1, \dots, l\}$. 
\end{enumerate}
\end{proposition}

\begin{proof}
A section $h$ of $\Sym^2TM^*$ is tangent to $\Met_{\perp}$ if and only if $h$ can be written as the sum $h=h_1\oplus \cdots \oplus  h_l$ of sections $h_i$ of $\Sym^2E_i^*$, $i\in\{1, \dots, l\}$. 
Also, $h$ is tangent to $ [\![ g ]\!]$ if and only if $h=\varphi_1g_1\oplus \cdots\oplus \varphi_lg_l$ for some $\varphi_i\in C^{\infty}(M)$, $i\in\{1, \dots, l\}$. 
By \eqref{eq:FundamentalInclusion} and \eqref{eq:RIsSimplyTheSum}, we may assume without loss of generality that $g$ has constant scalar curvature. 

We show (1). 
Assume $g$ is critical for $\left.E\right|_{ [\![ g ]\!]}$. 
Then, for each $i\in\{1, \dots, l\}$, 
\begin{align*}
0
&=\left.\frac{d}{dt}\right|_{t=0}E(g+t\varphi_ig_i)\\
&=-\int_M\langle \varphi_ig_i, \Ric^g-(R^g/m)g\rangle d\mu^g
=-\int_M\varphi_i\left(R^g_i-(R^g/m)m_i\right)d\mu^g
\end{align*}
for all $\varphi_i\in C^{\infty}(M)$ by \eqref{eq:FirstVariationFormulaEinsteinHilbert}, whence $R^g_i/m_i=R^g/m$. 
Conversely, assume $R^g_1/m_1=\cdots=R^g_l/m_l=c$. 
Then $R^g_i/m_i=R^g/m$ necessarily holds by \eqref{eq:RIsSimplyTheSum}, and 
\begin{align*}
\left.\frac{d}{dt}\right|_{t=0}E\left(g+t(\varphi_1g_1\oplus \cdots\oplus \varphi_lg_l)\right)
&=-\int_M\langle \varphi_1g_1\oplus \cdots\oplus \varphi_lg_l, \Ric^g-(R^g/m)g\rangle d\mu^g\\
&=-\sum_{i=1}^l\int_Mm_i\varphi_i(R^g_i/m_i-R^g/m)d\mu^g=0
\end{align*}
for all $\varphi_1, \dots, \varphi_l\in C^{\infty}(M)$ by \eqref{eq:FirstVariationFormulaEinsteinHilbert}. 
That is, $g$ is critical for $\left.E\right|_{ [\![ g ]\!]}$. 

We show (2). 
Assume $g$ is critical for $\left.E\right|_{\Met_{\perp}}$. 
Then, for each $i\in\{1, \dots, l\}$, 
\begin{align*}
0=\left.\frac{d}{dt}\right|_{t=0}E(g+th_i)=-\int_M\langle h_i, \Ric^g-(R^g/m)g\rangle d\mu^g
\end{align*}
for all $h_i\in\Gamma(\Sym^2E_i)$ by \eqref{eq:FirstVariationFormulaEinsteinHilbert}, whence $0=\Ric^g_i-(R^g/m)g_i$. 
Conversely, assume $\Ric^g_i=cg_i$ for all $i\in\{1, \dots, l\}$. Then $c=R^g_1/m_1=\cdots=R^g_l/m_l$, so $c=R^g/m$ by \eqref{eq:RIsSimplyTheSum}. Therefore, 
\begin{align*}
\left.\frac{d}{dt}\right|_{t=0}E\left(g+t(h_1\oplus \cdots\oplus  h_l)\right)
&=-\int_M\langle h_1\oplus \cdots\oplus  h_l, \Ric^g-(R^g/m)g\rangle d\mu^g\\
&=-\sum_{i=1}^l\int_M\langle h_i, \Ric^g_i-(R^g/m)g_i\rangle d\mu^g
=0
\end{align*}
for all $h_i\in\Gamma(\Sym^2E_i)$, $i\in\{1, \dots, l\}$. 
That is, $g$ is critical for $\left.E\right|_{\Met_{\perp}}$.  
\end{proof}

It is well known that, provided $m\ge 3$, 
\begin{align*}
&\inf_{g\in\Met}E(g)=-\infty, 
\qquad
\sup_{\tilde{g}\in[g]}E(\tilde{g})=\infty, \\
&-\infty<\inf_{\tilde{g}\in[g]}E(\tilde{g})\le m(m-1)\Vol(S^m(1))^{2/m}=\mu(S^m(1)).  
\end{align*}
Recall that the conformal Yamabe constant $\mu(M, [g])$ and the differential Yamabe invariant $\sigma(M)$, also known as Schoen's $\sigma$-constant, are defined respectively by 
\begin{align*}
&\mu(M, [g])=\inf_{\tilde{g}\in[g]}E(\tilde{g}), 
&
&\sigma(M)=\sup_{[g]\subset\Met}\mu(M, [g]). 
\intertext{In view of \eqref{eq:FundamentalInclusion}, we define}
&\sigma(M,  [\![ g ]\!])=\sup_{[\tilde{g}]\subset [\![ g ]\!]}\mu(M, [\tilde{g}]), 
&
&\sigma_{\perp}(M)=\sup_{[g]\subset\Met_{\perp}}\mu(M, [g])
\end{align*}
so that 
\begin{equation}\label{eq:FundamentalIneq}
-\infty<\mu(M, [g])\le\sigma(M,  [\![ g ]\!])\le\sigma_{\perp}(M)\le\sigma(M)\le \sigma(S^m). 
\end{equation}
We remark that $\mu(M, [g])>0$, $\sigma(M,  [\![ g ]\!])>0$, $\sigma_{\perp}(M)>0$, and $\sigma(M)>0$ hold, respectively, if and only if $[g]$, $ [\![ g ]\!]$, $\Met_{\perp}$, and $\Met$ contain a metric of positive scalar curvature. 

By definition, $\sigma_{\perp}(M)$ and $\sigma(M,  [\![ g ]\!])$ are invariants of almost product manifolds and multiconformal manifolds, respectively. 
Here, a morphism $\varphi: M\to N$ of almost product manifolds $(M, \ TM=E_1+\cdots+E_l)$ and $(N, \ TN=F_1+\cdots+F_l) $ is a smooth map such that $d\varphi(E_i)\subset d\varphi(F_i)$ for all $i\in\{1, \dots, l\}$. 
Also, two almost product Riemannian manifolds $(M, g=g_1\oplus \cdots\oplus  g_l)$ and $(N, h=h_1\oplus \cdots\oplus  h_l)$ are said to be multiconformally diffeomorphic if there exists an isomorphism $\varphi: M\to N$ of the underlying almost product manifolds so that $g$ and $\varphi^*h$ are multiconformally equivalent. 

The invariants $\sigma(M,  [\![ g ]\!])$ and $\sigma_{\perp}(M)$ have a resemblance in spirit 
to the equivariant Yamabe constant or invariant (cf.~B\'erard-Bergery \cite{Bergery_Kaigai}, Hebey--Vaugon \cite{MR1216009}) 
defined for manifolds with group actions. 
However, since an equivariant conformal class is smaller than the ordinary conformal class, 
one cannot expect an inequality like \eqref{eq:FundamentalIneq} for the equivariant ones (cf.~Ammann--Madani--Pilca \cite[Example 3]{MR3712199}). 

\section{Change of dependent variables}\label{sect:KillFirst}
Let $(M, g=\langle\cdot, \cdot\rangle)$ be a Riemannian manifold, $E$ a vector subbundle of $TM$, and $\mathcal{P}\in\Gamma(\End TM)$ the corresponding orthogonal projection. 
For $f\in C^{\infty}(M)$, 
we define $d_E f\in\Gamma(TM^*)$, $\grad^g_Ef\in\Gamma(TM)$, 
$\Hess^g_Ef\in\Gamma(\Sym^2TM^*)$, $\Delta^g_Ef\in C^{\infty}(M)$ by 
\begin{align}\label{eq:DifferentialOperators}\begin{split}
&d_Ef(X)=df(\mathcal{P}X),  \quad \grad^g_Ef=\mathcal{P}\grad^g f, \\
&\Hess^g_Ef(X, Y)=\Hess^gf(\mathcal{P}X, \mathcal{P}Y),  \quad \Delta^g_Ef =\tr^g\Hess^g_E
\end{split}\end{align}
for all $X, Y\in\Gamma(TM)$. 
Our sign conventions for $\Hess^g$ and $\Delta^g$ are the ones such that $\Hess f=f''dt\otimes dt$ and $\Delta f=f''$ on $(\real, dt^2)$. 
Note that the chain rules 
\begin{align*}
d_E(\tau\circ f)&=(\tau'\circ f)d_Ef, \\
\grad_E(\tau\circ f)&=(\tau'\circ f)\grad_E f, \\
\Hess^g_E(\tau\circ f)&=(\tau'\circ f)\Hess^g_Ef+(\tau''\circ f)d_Ef\otimes d_Ef, \\
\Delta^g_E(\tau\circ f)&=(\tau'\circ f)\Delta^g_Ef+(\tau''\circ f)\lvert d_Ef\rvert^2
\end{align*}
hold for all smooth $\tau: \real\to\real$, $f: M\to\real$ and that 
\begin{align*}
\langle d_Eu, d_Ev\rangle=\langle\grad^g_Eu, \grad^g_Ev\rangle
\end{align*}
for all $u, v\in C^{\infty}(M)$. 

Let $f_1, \dots, f_l : M\to\real$ be smooth functions and $a_j$, $b_{jk}$ real numbers such that $b_{jk}=b_{kj}$ for all $j, k\in \{1, \dots, l\}$, and consider the section
\begin{align*}
{\color{orange}\sum_{j=1}^la_j\frac{\Hess^g_E f_j}{f_j}+\sum_{j, k=1}^lb_{jk}\frac{d_Ef_j\otimes d_Ef_k}{f_j f_k}}
\end{align*}
of $\Sym^2TM^*$. 
We introduce $u_j:=\log f_j$ so that 
\begin{align*}
&{\color{orange}\sum_{j=1}^la_j\frac{\Hess^g_E f_j}{f_j}+\sum_{j, k=1}^lb_{jk}\frac{d_Ef_j\otimes d_Ef_k}{f_j f_k}}\\
&={\color{magenta}\sum_{j=1}^la_j\Hess^g_Eu_j
	+\sum_{j=1}^l(a_j+b_{jj})d_Eu_j\otimes d_Eu_j
		+\sum_{j\neq k}b_{jk}d_Eu_j\otimes d_Eu_k}.  
\end{align*}
We set 
\begin{align}\label{eq:matricesLatin}\begin{split}
\boldsymbol{a}&:=
	\begin{pmatrix}
	a_1 & \cdots & a_l
	\end{pmatrix}^T, 
\\
A&:=
	\begin{pmatrix}
	a_1&&\\
	&\ddots&\\
	&&a_l
	\end{pmatrix}, 
\quad B:=(b_{jk})=
	\begin{pmatrix}
	b_{11}&\cdots&b_{1l}\\
	\vdots&\ddots&\vdots\\
	b_{l1}&\cdots&b_{ll}
	\end{pmatrix}, 
\end{split}\end{align}
take an orthogonal matrix $P=(p_{\alpha j})$, and define
\begin{align}
\label{eq:matricesGreek}\boldsymbol{\kappa}&:=
	\begin{pmatrix}
	\kappa_1 & \cdots & \kappa_l
	\end{pmatrix}^T
=P\boldsymbol{a}, 
& & &
\Lambda&:=
(\lambda_{\alpha\beta})
=P(A+B)P^{-1}. 
\end{align}
We introduce $\psi_{\alpha}:=\sum_{j=1}^lp_{\alpha j}u_j$ for all $\alpha\in\{1, \dots, l\}$ so that $u_j=\sum_{\alpha=1}^lp_{\alpha j}\psi_{\alpha}$ and 
\begin{align*}
&{\color{magenta}\sum_{j=1}^la_j\Hess^g_Eu_j
	+\sum_{j=1}^l(a_j+b_{jj})d_Eu_j\otimes d_Eu_j
		+\sum_{j\neq k}b_{jk}d_Eu_j\otimes d_Eu_k}\\
&=\sum_{j=1}^la_j\Hess^g_Eu_j
	+\sum_{j, k=1}^l(A+B)_{jk}d_Eu_j\otimes d_Eu_k\\
&=\sum_{j, \alpha=1}^la_jp_{\alpha j}\Hess^g_E\psi_{\alpha}
	+\sum_{j, k, \alpha, \beta=1}^l(A+B)_{jk}p_{\alpha j}p_{\beta k}d_E\psi_{\alpha}\otimes d_E\psi_{\beta}\\
&={\color{blue}\sum_{\alpha=1}^l\kappa_{\alpha}\Hess^g_E\psi_{\alpha}+\sum_{\alpha, \beta=1}^l\lambda_{\alpha\beta}d_E\psi_{\alpha}\otimes d_E\psi_{\beta}}.  
\end{align*}
Taking the trace of both sides, we obtain the following formula for the change of dependent variables. 
\begin{lemma}\label{lem:ChangeOfVariables}
If  $\psi_{\alpha}=\sum_{j=1}^lp_{\alpha j}\log f_j$, then 
\begin{align*}
{\color{orange}\sum_{j=1}^la_j\frac{\Delta^g_E f_j}{f_j}+\sum_{j, k=1}^lb_{jk}\frac{\langle d_Ef_j, d_Ef_k\rangle}{f_j f_k}}
&={\color{blue}\sum_{\alpha=1}^l\kappa_{\alpha}\Delta^g_E\psi_{\alpha}+\sum_{\alpha, \beta=1}^l\lambda_{\alpha\beta}\langle d_E\psi_{\alpha}, d_E\psi_{\beta}\rangle}. 
\end{align*}
Here, $B=(b_{jk})$ is symmetric, $P=(p_{\alpha j})$ is orthogonal, and the coefficients satisfy the relations \eqref{eq:matricesLatin}, \eqref{eq:matricesGreek}. 
\end{lemma}
\noindent
In particular, when $B=(b_{jk})$ is negative (resp.~positive) definite, $\sum_{j, k=1}^lb_{jk}\frac{\langle d_Ef_j, d_Ef_k\rangle}{f_j f_k}$ is nonpositive (resp.~nonnegative) and is zero if and only if $d_Ef_1=\cdots=d_Ef_l=0\in\Gamma(TM^*)$.  

\section{Scalar curvature computation \`a la Karcher}\label{sect:Karcher}

Let $M^m=M_1^{m_1}\times\cdots\times M_l^{m_l}$ be a direct product of smooth manifolds. 
We denote by $E_i\to M$ the vector subbundle of $TM$ defined as the pullback of $TM_i$ via the projection $M\to M_i$, 
so that $TM=E_1+\cdots+E_l$ and $E_i\cap E_j$ is trivial if $i\neq j$ and $\mathcal{P}_i: TM \rightarrow E_i$ are the corresponding projections.
We adopt the notation and terminology from Section~\ref{sect:EinsteinHilbertRestrictions} in this special case. 
Following \eqref{eq:DifferentialOperators}, we define $\grad^g_i:=\grad^g_{E_i}$, $\Hess^g_i:=\Hess^g_{E_i}$, $\Delta^g_i:=\Delta^g_{E_i}$ for an arbitrary compatible metric $g$, so that 
\begin{align*}
\grad^g=\sum_{i=1}^l\grad^g_i, &&\Delta^g=\sum_{i=1}^l\Delta^g_i. 
\end{align*}

In this section, we derive the formula for the scalar curvature under a multiconformal change (Theorem \ref{thm:MulticonformalDeformation}). 
The strategy is the same as that of Karcher \cite{MR1669359}. Hence, we fix a compatible $g = g_1 \oplus \ldots \oplus g_l$ on $M= M_1 \times \ldots \times M_l$ and a multiconformal change $\tilde{g} = f_1^2 g_1 \oplus \ldots \oplus f_l^2 g_l$.
In what follows, we write $\tilde{g}=\langle\!\langle\cdot, \cdot\rangle\!\rangle$, $g=\langle\cdot, \cdot\rangle$ whenever convenient, with the respective norms $|\!|\cdot|\!|$, $|\cdot|$. In the proof, we also compute the differences $\Ric^{\tilde{g}}_i-\Ric^g_i$, see Eq.~\eqref{eq:Ricj}, and $R^{\tilde{g}}_i-R^g_i/f_i^2$, see Eq.~\eqref{eq:Rj}. 

\begin{remark}
When $g$ is a direct product metric, 
the multiconformally related metric $\tilde{g}$ is also called a \emph{twisted product}\footnote{
	cf.~Koike \cite[p.~3]{MR1212293}, Meumertzheim--Reckziegel--Schaaf \cite[Definition 2]{MR1726218}.  
}. If in addition $l=2$, it is more common to say that $\tilde{g}$ is \emph{biconformal}\footnote{
	cf.~Mo \cite[p.~15]{MR1735681}, 
	Slobodeanu \cite{MR2796655,MR2156024}, 
	Ou \cite[Lemma 2.1]{MR2171890}, 
	Danielo \cite[D\'efinition 2.1]{MR2246414}, 
	Rovenski--Zelenko \cite[p.~504]{MR3401903}.  
} to $g$. 
Moreover, for a direct product metric $g$, a formula without proof for the curvature tensor of $\tilde{g}$ can be found in 
Meumertzheim--Reckziegel--Schaaf \cite[Proposition 1]{MR1726218}. 
We present here the detailed computation, which works for an arbitrary compatible metric $g$ on $M=M_1\times\cdots\times M_l$.  
\end{remark}

Before we start computing the scalar curvature under a multiconformal change, we first relate the Levi-Civita connections $\nabla^g$, $\nabla^{\tilde{g}}$ of $g$, $\tilde{g}$, respectively. As the derivatives of the multiconformal factors $f_1, \ldots, f_l$ will be involved it is reasonable to first compare the gradients taken with respect to the metrics $g$ and $\tilde{g}$.

\begin{lemma}
For every $\varphi\in C^{\infty}(M)$, 
\begin{align}\label{eq:GradTilde}
{\color{blue}\grad^{\tilde{g}}\varphi}={\color{orange}\sum_{a=1}^lf_a^{-2}\grad^g_a\varphi}. 
\end{align}
\end{lemma}

\begin{proof}
For every $X\in\Gamma(TM)$, we can express the derivative $X(\varphi)$ either with respect to $g$ or with respect to $\tilde{g}$. This leads to,
\begin{align*}
X(\varphi) &= \langle\!\langle{\color{blue}\grad^{\tilde{g}}\varphi}, X\rangle\!\rangle, \\
X(\varphi) &= \langle\grad^g\varphi, X\rangle =\sum_{a=1}^l\langle\!\langle {\color{orange}f_a^{-2}\mathcal{P}_a\cdot\grad^g\varphi}, X\rangle\!\rangle.  
\end{align*}
As $X(\varphi)$ is independent of the Riemannian metric, \eqref{eq:GradTilde} holds.  
\end{proof}

\begin{proposition}\label{prop:DifferenceConnection}
Define $T_XY=\nabla^{\tilde{g}}_XY-\nabla^{g}_XY$ for $X, Y\in\Gamma(TM)$.  Then
\begin{align}\label{eq:DifferenceConnection}\begin{split}
T_XY
&=\sum_{a=1}^l \langle X, \grad^g f_a\rangle\frac{1}{f_a}\mathcal{P}_a Y
+\sum_{a=1}^l \langle Y, \grad^g f_a\rangle\frac{1}{f_a}\mathcal{P}_a X\\
&\quad-\sum_{a, b=1}^l\langle \mathcal{P}_b X, \mathcal{P}_b Y\rangle \frac{f_b}{f_a^2}\mathcal{P}_a\grad^g f_b.  
\end{split}\end{align}
\end{proposition}

\begin{proof}
Since $T$ is tensorial, we assume without loss of generality that $X\in\Gamma(TM_i)$, $Y\in\Gamma(TM_j)$.  
Comparing the Koszul formulas 
\begin{align*}
2\langle\!\langle\nabla^{\tilde{g}}_XY, Z\rangle\!\rangle&=X\langle\!\langle Y, Z\rangle\!\rangle+Y\langle\!\langle X, Z\rangle\!\rangle-Z\langle\!\langle X, Y\rangle\!\rangle\\
&\quad-\langle\!\langle X, [Y, Z]\rangle\!\rangle-\langle\!\langle Y, [X, Z]\rangle\!\rangle+\langle\!\langle Z, [X, Y]\rangle\!\rangle, \\
2\langle\nabla^{g}_XY, Z\rangle&=X\langle Y, Z\rangle+Y\langle X, Z\rangle-Z\langle X, Y\rangle\\
&\quad-\langle X, [Y, Z]\rangle-\langle Y, [X, Z]\rangle+\langle Z, [X, Y]\rangle
\end{align*}
for $\tilde{g}$ and $g$, we observe that if $Z\in\Gamma(TM_a)$ then
\begin{align*}
2f_a^2\langle\nabla^{\tilde{g}}_XY, Z\rangle
&=2\langle\!\langle\nabla^{\tilde{g}}_XY, Z\rangle\!\rangle\\
&=X\left( f_a^2\langle Y, Z\rangle\right)
+Y\left( f_a^2\langle X, Z\rangle\right)
-Z\langle\!\langle X, Y\rangle\!\rangle\\
&\quad-f_a^2\langle X, [Y, Z]\rangle
-f_a^2\langle Y, [X, Z]\rangle
+f_a^2\langle Z, [X, Y]\rangle\\
&=2f_aX (f_a)\langle Y, Z\rangle+2f_aY( f_a)\langle X, Z\rangle-Z\langle\!\langle X, Y\rangle\!\rangle\\
&{\color{blue}\quad+f_a^2X\langle Y, Z\rangle
+f_a^2Y \langle X, Z\rangle}\\
&{\color{blue}\quad-f_a^2Z \langle X, Y\rangle}+f_a^2Z \langle X, Y\rangle&&(=0)\\
& {\color{blue}\quad-f_a^2\langle X, [Y, Z]\rangle
-f_a^2\langle Y, [X, Z]\rangle
+f_a^2\langle Z, [X, Y]\rangle}\\
&={\color{blue}2f_a^2\langle\nabla^{g}_XY, Z\rangle}\\
&\quad+2f_aX(f_a)\langle Y, Z\rangle+2f_aY(f_a)\langle X, Z\rangle\\
&\quad-Z\langle\!\langle X, Y\rangle\!\rangle+f_a^2Z \langle X, Y\rangle.  
\end{align*}
Dividing both sides by $2f_a^2$, we obtain
\begin{align*}
\langle\nabla^{\tilde{g}}_XY-\nabla^{g}_XY, Z\rangle
&=\frac{X(f_a)}{f_a}\langle Y, Z\rangle+\frac{Y(f_a)}{f_a}\langle X, Z\rangle
-\frac{1}{2}\frac{Z}{f_a^2}\langle\!\langle X, Y\rangle\!\rangle+\frac{1}{2}Z \langle X, Y\rangle.  
\end{align*}
Hence,
\begin{align*}
T_XY
&=\nabla^{\tilde{g}}_XY-\nabla^g_XY\\
&=\sum_{a=1}^l \frac{X(f_a)}{f_a}\mathcal{P}_a Y
+\sum_{a=1}^l \frac{Y(f_a)}{f_a}\mathcal{P}_a X
-\frac{1}{2}\grad^{\tilde{g}}\langle\!\langle X, Y\rangle\!\rangle+\frac{1}{2}\grad^g\langle X, Y\rangle\\
&={\color{orange}\sum_{a=1}^l \langle X, \grad^g f_a\rangle\frac{1}{f_a}\mathcal{P}_a Y
+\sum_{a=1}^l \langle Y, \grad^g f_a\rangle\frac{1}{f_a}\mathcal{P}_a X}\\
&\quad-\frac{1}{2}\grad^{\tilde{g}}\langle\!\langle X, Y\rangle\!\rangle+\frac{1}{2}\grad^g\langle X, Y\rangle.  
\end{align*}

To conclude the claimed formula it remains to show that
\begin{align}\label{eq:StrangeLookingQuantity}
-\frac{1}{2}\grad^{\tilde{g}}\langle\!\langle X, Y\rangle\!\rangle+\frac{1}{2}\grad^g\langle X, Y\rangle
=-\sum_{a, b=1}^l\langle \mathcal{P}_b X, \mathcal{P}_b Y\rangle \frac{f_b}{f_a^2}\mathcal{P}_a\grad^g f_b
\end{align}
for $X, Y\in\Gamma(TM)$. 

The term $-\frac{1}{2}\grad^{\tilde{g}}\langle\!\langle X, Y\rangle\!\rangle+\frac{1}{2}\grad^g\langle X, Y\rangle$ is tensorial, since it is  the difference of two tensors $T_XY$ and 
${\color{orange}\sum_{a=1}^l \langle X, \grad^g f_a\rangle\frac{1}{f_a}\mathcal{P}_a Y
+\sum_{a=1}^l \langle Y, \grad^g f_a\rangle\frac{1}{f_a}\mathcal{P}_a X}$. 
Thus, it suffices to check \eqref{eq:StrangeLookingQuantity} for an $g$-orthonormal frame $\{e_{\alpha}\}_{\alpha=1}^m$ 
such that for every $\alpha\in\{1, \dots, m\}$ there exists some $i\in\{1, \dots, l\}$ with $e_{\alpha}\in\Gamma(E_i)$.  
Since $\{e_{\alpha}\}_{\alpha=1}^m$ remains orthogonal with respect to $\tilde{g}$, 
\eqref{eq:StrangeLookingQuantity} holds for $X=e_{\alpha}$, $Y=e_{\beta}$ if $\alpha\neq\beta$ as both sides evaluate to $0$.  
If $\alpha=\beta$, we derive for the left hand side
\begin{align*}
-\frac{1}{2}\grad^{\tilde{g}}\langle\!\langle e_{\alpha}, e_{\alpha}\rangle\!\rangle+\frac{1}{2}\grad^g\langle e_{\alpha}, e_{\alpha}\rangle
&=-\frac{1}{2}\grad^{\tilde{g}}(f_i^2)\\
&=-\frac{1}{2}\sum_{a=1}^lf_a^{-2}\mathcal{P}_a\grad^g(f_i^2)&&\text{$\because$ \eqref{eq:GradTilde}}\\
&=-{\color{blue}\sum_{a=1}^l\frac{f_i}{f_a^2}\mathcal{P}_a\grad^g f_i},
\end{align*}
and for the right hand side
\begin{align*}
-\sum_{a, b=1}^l\langle \mathcal{P}_b e_{\alpha}, \mathcal{P}_b e_{\alpha}\rangle \frac{f_b}{f_a^2}\mathcal{P}_a\grad^g f_b
&=-{\color{blue}\sum_{a=1}^l\frac{f_i}{f_a^2}\mathcal{P}_a\grad f_i}.  
\end{align*}
Now \eqref{eq:StrangeLookingQuantity} follows by combining these two identities.  
\end{proof}

Using the tensor $T= \nabla^{\tilde{g}} - \nabla^g$ and the above lemma we are now able to derive a formula for the scalar curvature under a multiconformal change.

\begin{theorem}\label{thm:MulticonformalDeformation}
The scalar curvature of the metric $\tilde{g}$ satisfies
\begin{align}\label{eq:ScalarCurvatureDifference}\begin{split}
&R^{\tilde{g}}-\sum_{i}\frac{R_i^g}{f_i^2}\\
&=-2\sum_{i}(m_i-1)\frac{\Delta^g_if_i}{f_i^3}
	-2\sum_{i\neq j}m_j\frac{\Delta^g_if_j}{f_i^2f_j}\\
&\quad-\sum_{i}(m_i-1)(m_i-4)\frac{\lvert\grad^g_if_i\rvert^2}{f_i^4}
	-2\sum_{i\neq j}m_j(m_i-2)\frac{\langle\grad^g_i f_i, \grad^g_i f_j\rangle}{f_i^3f_j}\\
&\quad-\sum_{i\neq j}m_j(m_j-1)\frac{\lvert\grad^g_if_j\rvert^2}{f_i^2f_j^2}
	-\sum_{i\neq j, j\neq k, k\neq i}m_jm_k\frac{\langle \grad^g_i f_j, \grad^g_i f_k\rangle}{f_i^2f_jf_k}, 
\end{split}\end{align}
where the indices $i, j, k$ run over $\{1, \dots, l\}$.  
\end{theorem}

\begin{remark}
Setting $f_1=\dots=f_l=:f$ in \eqref{eq:ScalarCurvatureDifference} yields the well-known formula
for the scalar curvature of the conformally deformed metric $\tilde{g}=f^2 g$. 
Also, setting $f_1\equiv 1$ and $f_i\in C^{\infty}_+(M_1)$ for all $i>1$ yields the scalar curvature formula for the multiply warped product metric $g_1\oplus f_2^2g_2\oplus\cdots \oplus f_l^2g_l$ (cf.~Dobarro--\"Unal \cite[Proposition 2.6]{MR2157416}). 
\end{remark}


\begin{proof}[Proof of Theorem \ref{thm:MulticonformalDeformation}]
To begin with, we recall the general formula 
\begin{align}\label{eq:CurvatureDifference}
 R^{\tilde{g}}(X, Y)Z
=R(X, Y)Z+(\nabla_XT)_YZ-(\nabla_YT)_XZ+T_XT_YZ-T_YT_XZ
\end{align}
for $X, Y, Z\in\Gamma(TM)$, 
which can be shown by summing up the following three identities 
\begin{align*}
-\nabla^{\tilde{g}}_{[X, Y]}Z
&=-\nabla_{[X, Y]}Z-{\color{blue}T_{[X, Y]}Z}, \\
-\nabla^{\tilde{g}}_Y\nabla^{\tilde{g}}_XZ
&=-\nabla^{\tilde{g}}_Y(\nabla_XZ+T_XZ)\\
&=-\nabla_Y\nabla_XZ-(\nabla_YT)_XZ{\color{blue}-T_{\nabla_YX}Z}{\color{orange}-T_X(\nabla_YZ)-T_Y(\nabla_XZ)}-T_YT_XZ, \\
\nabla^{\tilde{g}}_X\nabla^{\tilde{g}}_YZ
&=\nabla_X\nabla_YZ+(\nabla_XT)_YZ{\color{blue}+T_{\nabla_XY}Z}{\color{orange}+T_Y(\nabla_XZ)+T_X(\nabla_YZ)}+T_XT_YZ.  
\end{align*}
Here and henceforth in the proof, quantities without any superscript such as $R$ and $\nabla$ are understood to be the ones with respect to $g$.  

We want to express all these identities in terms of the functions $f_1, \ldots, f_l$ and their derivatives. 
Let $X\in\Gamma(E_i)$, $Y\in\Gamma(E_j)$, $Z\in\Gamma(E_k)$. 
Then \eqref{eq:DifferenceConnection} yields 
\begin{align}
\begin{split}
	T_XY
	&= \langle X, \grad  f_j\rangle\frac{1}{f_j}Y
	+\langle Y, \grad  f_i\rangle\frac{1}{f_i}X-\langle X, Y\rangle \sum_{a=1}^l\frac{f_i}{f_a^2}\mathcal{P}_a \grad  f_i, 
\end{split}\label{eq:TXYmedium}\\
	T_YZ
&=\langle Y, \grad  f_k\rangle\frac{1}{f_k}Z
+\langle Z, \grad  f_j\rangle\frac{1}{f_j}Y
-\langle Y, Z\rangle\sum_{a=1}^l\frac{f_j}{f_a^2}\mathcal{P}_a \grad  f_j.  \label{eq:TYZ}
\end{align}
Up to interchanging the roles of $X$ and $Y$ there are two terms that we need to take care of. Namely, $(\nabla_X T)_YZ$ and $T_X T_Y Z$

On the one hand, \eqref{eq:TYZ} and \eqref{eq:DifferenceConnection} yields 
\begin{align*}
(\nabla _XT)_YZ
&=\nabla _X(T_YZ)-T_{\nabla _XY}Z-T_Y(\nabla _XZ)\\
&= \langle Y, \nabla_X \grad f_k \rangle \frac{1}{f_k} Z - \langle Y, \grad f_k \rangle \frac{X(f_k)}{f_k^2} Z  \\
& \quad + \langle Z, \nabla_X \grad f_j \rangle \frac{1}{f_j} Y - \langle Z, \grad f_j \rangle \frac{X(f_j)}{f_j} Y\\
&\quad- \langle Y,Z \rangle \sum_{a=1}^l \left( \frac{X(f_j)}{f_a^2} - \frac{2f_j X(f_a)}{f_a^3} \right) \mathcal{P}_a \grad f_j \\
&\quad - \langle Y, Z \rangle \sum_{a=1}^l \frac{f_j}{f_a^2} \mathcal{P}_a \grad f_j\\
&={\color{blue}\Hess  f_k(X, Y)\frac{1}{f_k}Z}
		{\color{blue}-\langle X, \grad  f_k\rangle\langle Y, \grad  f_k \rangle\frac{1}{f_k^2}Z}\\
	&\quad{\color{orange}+\Hess  f_j(X, Z)\frac{1}{f_j} Y}
		{\color{orange}-\langle X, \grad  f_j\rangle\langle Z, \grad  f_j\rangle\frac{1}{f_j^2}Y}\\
	&\quad{\color{olive}-\langle Y, Z\rangle\langle X, \grad  f_j\rangle\sum_{a=1}^l \frac{1}{f_a^2}\mathcal{P}_a \grad  f_j}\\
	&\quad{\color{olive}+2\langle Y, Z\rangle\sum_{a=1}^l \langle X, \grad  f_a\rangle\frac{f_j}{f_a^3}\mathcal{P}_a \grad  f_j}\\
		&\quad - \langle Y, Z\rangle\sum_{a=1}^l\frac{f_j}{f_a^2}\mathcal{P}_a \nabla _X\grad  f_j.  
\end{align*}
Taking the inner product with $W \in \Gamma(E_h)$ leads to, 
\begin{align}\label{eq:NablaXTYZW}\begin{split}
&\langle\nabla_XT_YZ , W\rangle\\
&={\color{blue}\langle Z, W\rangle\Hess  f_k(X, Y)\frac{1}{f_k}}
		{\color{blue}-\langle Z,  W\rangle\langle X, \grad  f_k\rangle\langle Y, \grad  f_k \rangle\frac{1}{f_k^2}}\\
	&\quad{\color{orange}+\langle Y, W\rangle\Hess  f_j(X, Z)\frac{1}{f_j}}
		{\color{orange}-\langle Y, W\rangle\langle X, \grad  f_j\rangle\langle Z, \grad  f_j\rangle\frac{1}{f_j^2}}\\
	&\quad{\color{olive}-\langle Y, Z\rangle\langle X, \grad  f_j\rangle\langle W, \grad  f_j\rangle\frac{1}{f_h^2}}\\
	&\quad{\color{olive}+2\langle Y, Z\rangle \langle X, \grad  f_h\rangle\langle W, \grad  f_j\rangle\frac{f_j}{f_h^3}}
		{\color{olive}-\langle Y, Z\rangle\Hess  f_j(X, W)\frac{f_j}{f_h^2}}.  
\end{split}\end{align}

On the other hand, plug \eqref{eq:TYZ} into \eqref{eq:TXYmedium} to get 
\begin{align*}
T_XT_YZ
&={\color{blue}\sum_{a=1}^l \langle X, \grad  f_a\rangle\frac{1}{f_a}\mathcal{P}_a  T_YZ}
{\color{orange}+\langle T_YZ, \grad  f_i\rangle\frac{1}{f_i}X}\\
&\quad{\color{olive}-\sum_{a=1}^l\langle X, T_YZ\rangle \frac{f_i}{f_a^2}\mathcal{P}_a \grad  f_i}\\
&={\color{blue}\langle X, \grad  f_k\rangle\langle Y, \grad  f_k\rangle\frac{1}{f_k^2}Z}
	{\color{blue}+\langle X, \grad  f_j\rangle\langle Z, \grad  f_j\rangle\frac{1}{f_j^2}Y}\\
&\quad{\color{blue}-\langle Y, Z\rangle\sum_{a=1}^l \langle X, \grad  f_a\rangle \frac{f_j}{f_a^3}\mathcal{P}_a \grad  f_j}\\
&\quad{\color{orange}+\langle Y,\grad  f_k\rangle\langle Z,\grad  f_i\rangle\frac{1}{f_if_k}X}
	{\color{orange}+\langle Y,\grad  f_i\rangle\langle Z, \grad  f_j\rangle\frac{1}{f_if_j}X}\\
&\quad{\color{orange}-\langle Y, Z\rangle\sum_{c=1}^l\langle \mathcal{P}_c \grad  f_j, \grad  f_i\rangle\frac{f_j}{f_c^2f_i}X}\\
&\quad{\color{olive}-\langle X, Z\rangle\langle Y, \grad  f_k\rangle\sum_{a=1}^l\frac{f_i}{f_a^2f_k}\mathcal{P}_a \grad  f_i}\\
&\quad{\color{olive}-\langle X, Y\rangle\langle Z, \grad  f_j\rangle\sum_{a=1}^l\frac{f_i}{f_a^2f_j}\mathcal{P}_a \grad  f_i}\\
&\quad{\color{olive}+\langle Y, Z\rangle\langle X, \grad  f_j\rangle\sum_{a=1}^l \frac{f_j}{f_a^2f_i}\mathcal{P}_a \grad  f_i}.  
\end{align*}
Taking the inner product with $W$, 
\begin{align}\label{eq:TXTYZW}\begin{split}
&\langle T_XT_YZ, W\rangle\\
&={\color{blue}\langle Z, W\rangle\langle X, \grad \!  f_k\rangle\langle Y, \grad \!  f_k\rangle\frac{1}{f_k^2}}
{\color{blue}+\langle Y, W\rangle\langle X, \grad \!  f_j\rangle\langle Z, \grad \!  f_j\rangle\frac{1}{f_j^2}}\\
&\quad{\color{blue}-\langle Y, Z\rangle\langle X, \grad \!  f_h\rangle\langle W, \grad \!  f_j\rangle \frac{f_j}{f_h^3}}\\
&\quad{\color{orange}+\langle X, W\rangle\langle Y,\grad \!  f_k\rangle\langle Z,\grad \!  f_i\rangle\frac{1}{f_if_k}}
{\color{orange}+\langle X, W\rangle\langle Y,\grad \!  f_i\rangle\langle Z, \grad \!  f_j\rangle\frac{1}{f_if_j}}\\
&\quad{\color{orange}-\langle X, W\rangle\langle Y, Z\rangle\sum_{c=1}^l\langle \grad \! _c f_j, \grad \! _c f_i\rangle\frac{f_j}{f_c^2f_i}}\\
&\quad{\color{olive}-\langle X, Z\rangle\langle Y, \grad \!  f_k\rangle\langle W, \grad \!  f_i\rangle\frac{f_i}{f_h^2f_k}}
{\color{olive}-\langle X, Y\rangle\langle Z, \grad \!  f_j\rangle\langle W, \grad \!  f_i\rangle\frac{f_i}{f_h^2f_j}}\\
&\quad{\color{olive}+\langle Y, Z\rangle\langle X, \grad \!  f_j\rangle\langle W, \grad \!  f_i\rangle\frac{f_j}{f_h^2f_i}}.  
\end{split}\end{align}

Therefore, \eqref{eq:CurvatureDifference}, \eqref{eq:NablaXTYZW}, and \eqref{eq:TXTYZW} yields a formula for the difference
\begin{align*}
\langle R^{\tilde{g}}(X,Y)Z - R(X,Y)Z,W \rangle
\end{align*}
for all $X \in \Gamma(E_i), Y \in \Gamma(E_j), Z \in \Gamma(E_k), W \in \Gamma(E_h)$ and by linearity, it extends to an identity for all vector fields on $M$. However, the resulting formula is a very long expression. As we are interested in a formula for the scalar curvature for a multiconformal change we only consider the difference $\langle R^{\tilde{g}}(X,Y)Y - R(X,Y)Y,X \rangle$ for $X \in \Gamma(E_i)$ and $Y \in \Gamma(E_j)$. In that case we obtain

\begin{align}
\begin{split}
\langle &R^{\tilde{g}}(X, Y)Y-R(X, Y)Y, X\rangle\\
& ={\color{orange}\langle X, Y\rangle\Hess f_j(X, Y)\frac{1}{f_j}}
	{\color{olive}+\langle X, Y\rangle\Hess f_i(Y, X)\frac{1}{f_i}}\\
	&\quad{\color{orange}-\langle X, X\rangle\Hess f_i(Y, Y)\frac{1}{f_i}}
	{\color{olive}-\langle Y, Y\rangle\Hess f_j(X, X)\frac{f_j}{f_i^2}}\\		
	&\quad\uwave{{\color{black}-4\langle X, Y\rangle\langle X, \grad f_i\rangle\langle Y, \grad f_j\rangle\frac{1}{f_if_j}}}\\
	&\quad{\color{olive}+2\langle Y, Y\rangle \langle X, \grad f_i\rangle\langle X, \grad f_j\rangle\frac{f_j}{f_i^3}}\\
	&\quad{\color{orange}+2\langle X, X\rangle\langle Y,\grad f_i\rangle\langle Y,\grad f_j\rangle\frac{1}{f_if_j}}\\
	&\quad{\color{orange}-\langle X, X\rangle\langle Y, Y\rangle\sum_{c=1}^l\langle \grad_c f_i, \grad_c f_j\rangle\frac{f_j}{f_c^2f_i}}\\
	&\quad{\color{orange}+\langle X, Y\rangle^2\sum_{c=1}^l\langle \grad_c f_i, \grad_c f_j\rangle\frac{f_i}{f_c^2f_j}}.  
\end{split}\label{eq:SectionalCurvatureDifference}
\end{align}

Taking an $g$-orthonormal frame $\{e_{\alpha}\}_{\alpha=1}^m$ 
so that for each $\alpha\in\{1, \dots, m\}$ there exists some $i=i(\alpha)\in\{1, \dots, l\}$ with $e_{\alpha}\in\Gamma(E_i)$, 
we define the associated $\tilde{g}$-orthonormal frame via $\{\tilde{e}_{\alpha}=f_{i(\alpha)}^{-1}e_{\alpha}\}_{\alpha=1}^m$.  
With respect to these orthonormal frames we conclude
\begin{align*}
\Ric^{\tilde{g}}(Y, Y)-\Ric(Y, Y)
&=\sum_{\alpha=1}^m\langle\!\langle R^{\tilde{g}}(\tilde{e}_{\alpha}, Y)Y, \tilde{e}_{\alpha}\rangle\!\rangle-\langle R(e_{\alpha}, Y)Y, e_{\alpha}\rangle\\
&=\sum_{\alpha=1}^m\langle  R^{\tilde{g}}(e_{\alpha}, Y)Y-R(e_{\alpha}, Y)Y, e_{\alpha}\rangle\\
&=\sum_{i=1}^l\sum_{\alpha}\langle R^{\tilde{g}}(e_{\alpha}, Y)Y-R(e_{\alpha}, Y)Y, e_{\alpha}\rangle.
\end{align*}
Inserting \eqref{eq:SectionalCurvatureDifference} leads to
\begin{align*}
&{\color{blue}\sum_{i=1}^l\sum_{\alpha}\langle  R^{\tilde{g}}(e_{\alpha}, Y)Y-R(e_{\alpha}, Y)Y, e_{\alpha}\rangle}\\
&={\color{orange}\sum_{i=1}^l\sum_{\alpha}\langle e_{\alpha}, Y\rangle\Hess f_j(e_{\alpha}, Y)\frac{1}{f_j}}
	{\color{olive}+\sum_{i=1}^l\sum_{\alpha}\langle e_{\alpha}, Y\rangle\Hess f_i(Y, e_{\alpha})\frac{1}{f_i}}\\
	&\quad{\color{orange}-\sum_{i=1}^l\sum_{\alpha}\Hess f_i(Y, Y)\frac{1}{f_i}}
	{\color{olive}-\sum_{i=1}^l\sum_{\alpha}\vert Y \vert^2 \Hess f_j(e_{\alpha}, e_{\alpha})\frac{f_j}{f_i^2}}\\		
	&\quad{\color{black}-4\sum_{i=1}^l\sum_{\alpha}\langle e_{\alpha}, Y\rangle\langle e_{\alpha}, \grad f_i\rangle\langle Y, \grad f_j\rangle\frac{1}{f_if_j}}\\
	&\quad{\color{olive}+2\sum_{i=1}^l\sum_{\alpha}\vert Y\vert^2 \langle e_{\alpha}, \grad f_i\rangle\langle e_{\alpha}, \grad f_j\rangle\frac{f_j}{f_i^3}}\\
	&\quad{\color{orange}+2\sum_{i=1}^l\sum_{\alpha}\langle Y,\grad f_i\rangle\langle Y,\grad f_j\rangle\frac{1}{f_if_j}}\\
	&\quad{\color{orange}-\sum_{i=1}^l\sum_{\alpha}\vert Y\vert^2 \sum_{c=1}^l\langle \grad_c f_i, \grad_c f_j\rangle\frac{f_j}{f_c^2f_i}}\\
	&\quad{\color{orange}+\sum_{i=1}^l\sum_{\alpha}\langle e_{\alpha}, Y\rangle^2\sum_{c=1}^l\langle \grad_c f_i, \grad_c f_j\rangle\frac{f_i}{f_c^2f_j}}\\
&=2\Hess f_j(Y, Y)\frac{1}{f_j}
	{\color{orange}-\sum_{i=1}^lm_i\Hess f_i(Y, Y)\frac{1}{f_i}}\\
	&\quad{\color{olive}-\sum_{i=1}^l\lvert Y\rvert^2\Delta_i f_j\frac{f_j}{f_i^2}}
	{\color{black}-4\frac{\langle Y, \grad f_j\rangle^2}{f_j^2}}\\
	&\quad{\color{olive}+2\sum_{i=1}^l\lvert Y\rvert^2\langle \grad_i f_i, \grad_i f_j\rangle\frac{f_j}{f_i^3}}
	{\color{orange}+2\sum_{i=1}^lm_i\langle Y,\grad f_i\rangle\langle Y,\grad f_j\rangle\frac{1}{f_if_j}}\\
	&\quad{\color{orange}-\sum_{i=1}^lm_i\lvert Y\rvert^2\sum_{c=1}^l\langle \grad_c f_i, \grad_c f_j\rangle\frac{f_j}{f_c^2f_i}}
	{\color{orange}+\lvert Y\rvert^2\sum_{c=1}^l\frac{\lvert\grad_c f_j\rvert^2}{f_c^2}}.  
\end{align*}
We thus obtain 
\begin{align}\label{eq:Ricj}
\begin{split}
\Ric^{\tilde{g}}_j-\Ric_j
&=2\frac{\Hess_j f_j}{f_j}{\color{orange}-\sum_{i=1}^lm_i\frac{\Hess_j f_i}{f_i}}{\color{olive}-\sum_{i=1}^l\Delta_i f_j\frac{f_j}{f_i^2}}g_j\\
	&\quad{\color{black}-4\frac{d_jf_j\otimes d_jf_j}{f_j^2}}{\color{orange}+2\sum_{i=1}^lm_i\frac{d_jf_i\otimes d_jf_j}{f_if_j}}\\
	&\quad{\color{olive}+2\sum_{i=1}^l\langle \grad_i f_i, \grad_i f_j\rangle\frac{f_j}{f_i^3}}g_j\\
	&\quad{\color{orange}-\sum_{i=1}^lm_i\sum_{c=1}^l\langle \grad_c f_i, \grad_c f_j\rangle\frac{f_j}{f_c^2f_i}}g_j
	{\color{orange}+\sum_{c=1}^l\frac{\lvert\grad_c f_j\rvert^2}{f_c^2}}g_j. 
\end{split}
\intertext{Taking the trace with respect to $g$ in \eqref{eq:Ricj} yields}
\label{eq:Rj}
\begin{split} R^{\tilde{g}}_j-\frac{R_j}{f_j^2}
&=2\frac{\Delta_j f_j}{f_j^3}
	{\color{orange}-\sum_{i=1}^lm_i\frac{\Delta_jf_i}{f_j^2f_i}}
	{\color{olive}-\sum_{i=1}^lm_j\frac{\Delta_i f_j}{f_i^2f_j}}\\		
	&\quad{\color{black}-4\frac{\lvert\grad_j f_j\rvert^2}{f_j^4}}
	{\color{orange}+2\sum_{i=1}^lm_i\frac{\langle\grad_jf_i, \grad_jf_j\rangle}{f_j^3f_i}}\\
	&\quad{\color{olive}+2\sum_{i=1}^lm_j\frac{\langle \grad_i f_i, \grad_i f_j\rangle}{f_i^3f_j}}\\
	&\quad{\color{orange}-\sum_{i=1}^lm_im_j\sum_{c=1}^l\frac{\langle \grad_c f_i, \grad_c f_j\rangle}{f_c^2f_if_j}}
	{\color{orange}+m_j\sum_{c=1}^l\frac{\lvert\grad_c f_j\rvert^2}{f_c^2f_j^2}}. 
\end{split}
\end{align}
Since $R^{\tilde{g}} = \sum_{j=1}^l R^{\tilde{g}}_j$ we sum \eqref{eq:Rj} over $j\in\{1, \dots, l\}$ and derive
\begin{align*}
&R^{\tilde{g}}-\sum_{j=1}^l\frac{R_j}{f_j^2}\\
&=2\sum_{j=1}^l\frac{\Delta_j f_j}{f_j^3}
	-2\sum_{i, j=1}^lm_i\frac{\Delta_jf_i}{f_j^2f_i}\\
	&\quad-4\sum_{j=1}^l\frac{\lvert\grad_j f_j\rvert^2}{f_j^4}
	+4\sum_{i, j=1}^lm_i\frac{\langle\grad_jf_i, \grad_jf_j\rangle}{f_j^3f_i}\\
	&\quad-\sum_{i, j=1}^lm_im_j\sum_{c=1}^l\frac{\langle \grad_c f_i, \grad_c f_j\rangle}{f_c^2f_if_j}
	+\sum_{j=1}^lm_j\sum_{c=1}^l\frac{\lvert\grad_c f_j\rvert^2}{f_c^2f_j^2}\\
&=-2\sum_{i=1}^l(m_i-1)\frac{\Delta_i f_i}{f_i^3}
	-2\sum_{i=1}^l\sum_{j\neq i}m_j\frac{\Delta_if_j}{f_i^2f_j}\\
	&\quad-\sum_{i=1}^l(m_i-1)(m_i-4)\frac{\lvert\grad_i f_i\rvert^2}{f_i^4}
	-2\sum_{i=1}^l\sum_{j\neq i}m_j(m_i-2)\frac{\langle\grad_if_i, \grad_if_j\rangle}{f_j^3f_i}\\
	&\quad-\sum_{i=1}^l\sum_{j\neq i}m_j(m_j-1)\frac{\lvert\grad_i f_j\rvert^2}{f_i^2f_j^2}
	-\sum_{i=1}^l\sum_{j\neq i, k\neq i, j\neq k}m_jm_k\frac{\langle \grad_i f_j, \grad_i f_k\rangle}{f_i^2f_jf_k}. 
\end{align*}
This is equivalent to the claimed formula \eqref{eq:ScalarCurvatureDifference} for the scalar curvature under a multiconformal change.
\end{proof}

\section{Integral and pointwise formulas}\label{sect:Formulas}
We summarize some formulas which will be necessary in the following  sections. 
Let $(M, g)=(M_1, g_1)\times\cdots\times(M_l, g_l)$ be a direct product Riemannian manifold and $\tilde{g}=f_1^2g_1\oplus\cdots\oplus f_l^2g_l$ multiconformal to $g$. 
Since we would like to apply the results of Section \ref{sect:KillFirst} to the identity for the scalar curvature of $(M,\tilde{g})$ derived in Theorem \ref{thm:MulticonformalDeformation}, it is convenient to write
\begin{align*}
R^{\tilde{g}}=\sum_{i=1}^l\frac{R^g_i}{f_i^2}+\sum_{i=1}^l\frac{\rho_i}{f_i^2}
\end{align*}
where $\rho_i$ is defined by 
\begin{align}\label{eq. Definition Rho}
\begin{split}
\rho_i&=\rho_i^g(f_1, \dots, f_l)\\
&=-2(m_i-1)\frac{\Delta^g_if_i}{f_i}
	-2\sum_{j\neq i}m_j\frac{\Delta^g_if_j}{f_j}\\
&\quad -(m_i-1)(m_i-4)\frac{{\color{blue}\lvert\grad^g_if_i\rvert^2}}{f_i^2}
	-2(m_i-2)\displaystyle\sum_{j\neq i}m_j\frac{{\color{orange}\langle\grad^g_i f_i, \grad^g_i f_j\rangle}}{f_if_j}\\
&\quad 	-\sum_{j\neq i}m_j(m_j-1)\frac{{\color{olive}\lvert\grad^g_if_j\rvert^2}}{f_j^2}
	-\sum_{j\neq i, k\neq i, j\neq k}m_jm_k\frac{{\color{magenta}\langle \grad^g_i f_j, \grad^g_i f_k\rangle}}{f_jf_k}, 
\end{split}
\end{align}
so that only derivatives in the direction of $M_i$ are involved. 
We observe that $\rho_i$ is invariant under rescalings; in other words, for all real numbers $c_1, \dots, c_l>0$, $\rho_i^g(c_1f_1, \dots, c_lf_l)=\rho_i^g(f_1, \cdots, f_l)$. 

For real numbers $q_1, \dots, q_l$ and for each $i\in\{1, \dots, l\}$ we consider the integral, 
\begin{align}\label{eq:IntegralRhoi}\begin{split}
&{\color{blue}\int_M\frac{\rho_i}{f_i^2}f_1^{q_1}\cdots f_l^{q_l}d\mu^g}\\
&=(m_i-1)(2q_i-m_i-2)\int_M\frac{{\color{blue}\lvert \grad^g_i f_i\rvert^2}}{f_i^4} f_1^{q_1}\cdots f_l^{q_l}d\mu^g\\
&\quad+\sum_{j\neq i}m_j(2q_j-m_j-1)\int_M\frac{{\color{olive}\lvert\grad^g_if_j\rvert^2}}{f_i^2f_j^2} f_1^{q_1}\cdots f_l^{q_l}d\mu^g\\
&\quad+2\sum_{j\neq i}(m_iq_j+q_im_j-m_im_j-q_j)\int_M\frac{{\color{orange}\langle \grad^g_i f_i, \grad^g_i f_j\rangle}}{f_i^3f_j} f_1^{q_1}\cdots f_l^{q_l}d\mu^g\\ 
&\quad+\sum_{j\neq i, k\neq i, j\neq k}(m_jq_k+q_jm_k-m_jm_k)
		\int_M\frac{{\color{magenta}\langle \grad^g_i f_j, \grad^g_i f_k\rangle}}{f_i^2f_jf_k}f_1^{q_1}\cdots f_l^{q_l}d\mu^g.
\end{split}\end{align}
To get rid of the second derivatives we need to assume that $g=g_1\oplus \cdots\oplus  g_l$ is a direct product metric in order to integrate by parts along each $M_i$ separately. Under this assumption we obtain
\begin{align}
\begin{split} \label{eq:IntegralByParts1}
&-2(m_i-1)\int_M\frac{\Delta^g_if_i}{f_i^3}f_1^{q_1}\cdots f_l^{q_l}d\mu^g\\
&\qquad =2(m_i-1)(q_i-3)\int_M\frac{{\color{blue}\lvert \grad^g_i f_i\rvert^2}}{f_i^4} f_1^{q_1} \cdots f_l^{q_l}d\mu^g\\
	&\qquad \quad+2(m_i-1)\sum_{j\neq i}q_j\int_M\frac{{\color{orange}\langle \grad^g_i f_i, \grad^g_i f_j\rangle}}{f_i^3f_j} f_1^{q_1} \cdots f_l^{q_l}d\mu^g
	\end{split}\\
\intertext{and}
\begin{split}\label{eq:IntegralByParts2}
& -2\sum_{j\neq i}m_j\int_M\frac{\Delta^g_if_j}{f_i^2f_j}f_1^{q_1}\cdots f_l^{q_l}d\mu^g\\
&\qquad =2(q_i-2)\sum_{j\neq i}m_j\int_M\frac{{\color{orange}\langle \grad^g_i f_i, \grad^g_i f_j\rangle}}{f_i^3f_j}f_1^{q_1}\cdots f_l^{q_l}d\mu^g\\
	&\qquad \quad+2\sum_{j\neq i}m_j(q_j-1)\int_M\frac{{\color{olive}\lvert\grad^g_i f_j\rvert^2}}{f_i^2f_j^2}f_1^{q_1}\cdots f_l^{q_l}d\mu^g\\
	&\qquad \quad+2\sum_{j\neq i, k\neq i, j\neq k}m_jq_k
		\int_M\frac{{\color{magenta}\langle \grad^g_i f_j, \grad^g_i f_k\rangle}}{f_i^2f_jf_k}f_1^{q_1}\cdots f_l^{q_l}d\mu^g. 
		\end{split} 
\end{align}

Inserting \eqref{eq:IntegralByParts1} and \eqref{eq:IntegralByParts2} into \eqref{eq:IntegralRhoi} it follows that for each $(q_1, \dots, q_l)\in\real^l$ the symmetric $(l\times l)$-matrix $B^i(q_1, \dots, q_l)$ defined by
\begin{align*}
&B^i(q_1, \dots, q_l)\\
&=\left(b^i_{jk}(q_1, \dots, q_l)\right)\\
&=-\left(\begin{smallmatrix}m_jm_k-m_jq_k-m_kq_j\end{smallmatrix}\right)_{jk}\\
&\quad-\left(\begin{smallmatrix}
m_1&&&(m_i+1)q_1-m_1q_i\\
&\ddots&&\vdots\\
&&m_{i-1}&(m_i+1)q_{i-1}-m_{i-1}q_i\\
(m_i+1)q_1-m_1q_i&\cdots&(m_i+1)q_{i-1}-m_{i-1}q_i&m_i+2q_i-2&(m_i+1)q_{i+1}-m_{i+1}q_i&\cdots&(m_i+1)q_l-m_lq_i\\
&&&(m_i+1)q_{i+1}-m_{i+1}q_i &m_{i+1}\\
&&&\vdots&&\ddots\\
&&&(m_i+1)q_l-m_lq_i&&&m_{l}
\end{smallmatrix}\right)
\end{align*}
satisfies
\begin{align}\label{eq:KeyIntegralFormula}\begin{split}
&\int_M\left(R^{\tilde{g}}-\sum_{i=1}^l\frac{R^g_i}{f_i^2}\right)f_1^{q_1}\cdots f_l^{q_l}d\mu_g\\
&=\sum_{i=1}^l\int_M\left(\sum_{j, k=1}^lb^i_{jk}\frac{\langle \grad^g_i f_j, \grad^g_i f_k\rangle}{f_jf_k}\right)\frac{f_1^{q_1}\cdots f_l^{q_l}}{f_i^2}d\mu_g. 
\end{split}\end{align}
Now we can apply the results of Section \ref{sect:KillFirst} to conclude that the right hand side is nonpositive (resp.\ nonnegative) if for all $1 \leq i \leq l$ the matrices $B^i$ are negative (resp.\ positive) definite. 

In Section \ref{sect:WarpedProduct} we consider multiconformal metrics of permutation type, see Definition \ref{def:PermutationType} and metrics of warped product type. Since in these cases the multiconformally changed metric $\tilde{g}= f_1^2 g_1\oplus \cdots \oplus f_l^2g_l$ is such that the functions $f_i$ are constants along some of the factors of the product manifold $M= M_1 \times \cdots \times M_l$ the integral formula \ref{eq:KeyIntegralFormula} simplifies. 

Let us assume that the metric $\tilde{g}= f_1^2 g_1\oplus \ldots \oplus f_l^2g_l$ is such that there is an $i \in \lbrace 1, \ldots, l \rbrace$ such that the functions $f_1, \ldots, f_l$ only depend on $M_i$, i.e.\ are constant along $M_1 \times \cdots \times M_l$.  In that case we fix a function $\varphi: M\to\real$ that is constant along $M_1\times\cdots\times M_{i-1}\times M_{i+1}\times\cdots\times M_l$. Then there are uniquely determined  functions $\alpha_1, \dots, \alpha_l: \real\to\real$ such that $f_j=\exp(\alpha_j\circ\varphi)$. 
Since 
\begin{align*}
\frac{d_kf_j}{f_j}=(\alpha_j'\circ\varphi)d_k\varphi, 
&&
\frac{\Delta^g_kf_j}{f_j}=(\alpha_j'\circ\varphi)\Delta^g_k\varphi+\left( \alpha_j''\circ\varphi+(\alpha_j'\circ\varphi)^2\right)\lvert \grad^g_k\varphi\rvert^2, 
\end{align*}
for $\tilde{g}=f_1^2g_1\oplus\cdots\oplus f_l^2g_l$,  \eqref{eq:ScalarCurvatureDifference} can be written as
\begin{align}\label{eq:KeyPointwiseFormula}\begin{split}
&f_i^2\left(R^{\tilde{g}}-\sum_{j=1}^l\frac{R_j^g}{f_j^2}\right)\\
&=-\left(2(m_i-1)\alpha_i'+\sum_{j\neq i}2m_j\alpha_j'\right)\Delta^g\varphi\\
&\quad-\left(2(m_i-1)\left(\alpha_i''+(\alpha_i')^2\right)+\sum_{j\neq i}2m_j\left(\alpha_j''+(\alpha_j')^2\right)\right.\\
&\quad\phantom{-}\quad+(m_i-1)(m_i-4)(\alpha_i')^2+\sum_{j\neq i}m_j(m_j-1)(\alpha_j')^2\\
&\quad\phantom{-}\left.\quad+\sum_{j\neq i}2m_j(m_i-2)\alpha_i'\alpha_j'+\sum_{j, k\neq i; j\neq k}m_jm_k\alpha_j'\alpha_k'\right)\lvert\grad^g\varphi\rvert^2.
\end{split}\end{align}
Hence, integral formula \eqref{eq:IntegralRhoi} with $(q_1, \dots, q_l)=(m_1, \dots, m_l)$ simplifies to
\begin{align*}
&\int_M\frac{\rho_i}{f_i^2}f_1^{m_1}\cdots f_l^{m_l}d\mu^g\\
&=(m_i-1)(m_i-2)\int_M(a_i'\circ\varphi)^2{\color{blue}\lvert \grad_i \varphi\rvert^2} \frac{f_1^{m_1}\cdots f_l^{m_l}}{f_i^2}d\mu^g\\
&\quad+\sum_{j\neq i}m_j(m_j-1)\int_M(a_j'\circ\varphi)^2{\color{olive}\lvert\grad_i\varphi\rvert^2}\frac{f_1^{m_1}\cdots f_l^{m_l}}{f_i^2}d\mu^g\\
&\quad+2(m_i-1)\sum_{j\neq i}m_j\int_M(a_i'\circ\varphi)(a_j'\circ\varphi){\color{orange}\lvert\grad_i\varphi\rvert^2} \frac{f_1^{m_1}\cdots f_l^{m_l}}{f_i^2}d\mu^g\\
&\quad+\sum_{j\neq i, k\neq i, j\neq k}m_jm_k
		\int_M(a_j'\circ\varphi)(a_k'\circ\varphi){\color{magenta}\lvert\grad_i\varphi\rvert^2}\frac{f_1^{m_1}\cdots f_l^{m_l}}{f_i^2}d\mu^g.
\end{align*}

\section{The sign of a multiconformal class}\label{sect:MulticonformalSign}

If $\dim(M) \leq 2$, we define its Yamabe constant as $\mu(M^2, [g])=4\pi\chi(M^2)$ if $\dim M=2$ and $\mu(M^1, [d\theta^2])=0$ if $\dim M=1$. 

\begin{theorem}\label{thm:SignOfMulticonformalClass}
Let $(M, g)=(M_1, g_1)\times\cdots\times(M_l, g_l)$ be a direct product of closed connected Riemannian manifolds such that $m_i \geq 2$ for all $1 \leq i \leq l$. 
\begin{enumerate}
\item $ [\![ g ]\!]$ contains a metric of positive scalar curvature if and only if $\mu(M_i, [g_i])>0$ for some $i$. 
\item $ [\![ g ]\!]$ does not contain a metric of positive scalar curvature but a scalar flat metric if and only if $\mu(M_i, [g_i])=0$ for every $i$. 	In this case, if $\tilde{g}\in [\![ g ]\!]$ has nonnegative scalar curvature, then $\tilde{g}$ is necessarily scalar flat and direct product. 
\item $ [\![ g ]\!]$ does not contain a metric of nonnegative scalar curvature if and only if $\mu(M_i, [g_i])\le 0$ for every $i$ and $\mu(M_i, [g_i])<0$ for some $i$. 
\end{enumerate}
\end{theorem}

\begin{proof}

Without loss of generality we assume that $g_i$ is a constant scalar curvature metric for all $1 \leq i \leq l$. By Theorem \ref{thm:MulticonformalDeformation} the scalar curvature of $\tilde{g} = f_1^2 g_1 \oplus \cdots \oplus f_l^2 g_l \in [\![ g ]\!]$ is given by
\begin{align*}
R^{\tilde{g}} = \sum_i \frac{R^{g_i} + \rho_i}{f_i^2},
\end{align*}
where $\rho_i$ is the scale-invariant function defined in \eqref{eq. Definition Rho} for any $1 \leq i \leq l$.

Considering \eqref{eq:KeyIntegralFormula} for $(q_1, \ldots, q_l) = (0, \ldots, 0)$ we obtain
\begin{align}\label{eq.Integralformula q=0}
\int_M \left(R^{\tilde{g}} - \sum_i \frac{R^{g_i}}{f_i^2} \right)d\mu_g \leq \sum_i \int_M \left( \sum_{j,k=1}^l b_{jk}^i \frac{\langle \grad_i^g f_j, \grad_i^g f_k \rangle}{f_j f_k}\right) \frac{1}{f_i^2} d\mu_g,
\end{align}
where $(b_{jk}^i)_{jk}$ are the entries of the symmetric matrix
\begin{align*}
B^i = -(m_jm_k)_{jk} -
\left(\begin{smallmatrix}
m_1 & & & & & & \\
& \ddots & & & & & \\
& & m_{i-1} & & & &  \\
& & & m_i -2 & & & \\
& & & & m_{i+1} & & \\
& & & & & \ddots & \\
& & & & & & m_l
\end{smallmatrix}\right). 
\end{align*}
Since $m_i \geq 2$ for all $ 1 \leq i \leq l$ the matrix $B^i$ is negative definite for any $1 \leq i \leq l$. To see this, let $x \in \real^l$. Then
\begin{align*}
\langle x, B^i x \rangle = -\left( \sum_j m_j x_j \right)^2 - \sum_j (m_j - 2 \delta_{ij}) x_j^2.
\end{align*}
We note that the first summand is nonpositive and $0$ if and only if $(x_1, \ldots, x_l) \perp (m_1, \ldots, m_l)$ while the second summand is nonpositive and $0$ if and only if $m_i =2$ and $x= (0, \ldots, 0, x_i, 0, \ldots, 0)$. As these both sets are disjoint it follows that $B^i$ is negative definite. 

Applying Lemma \ref{lem:ChangeOfVariables} we conclude that the right-hand side of \eqref{eq.Integralformula q=0} is nonpositive and $0$ if and only if $\grad^g f_i = 0$ for all $1 \leq i \leq l$. In particular,
\begin{align}\label{eq:KeqInequality}
\int_M \left(R^{\tilde{g}} - \sum_i \frac{R^{g_i}}{f_i^2}\right)d\mu^g \leq 0
\end{align}
where equality holds if and only if $f_1, \ldots, f_l$ are all constant.

Using this inequality we can now prove the statements of Theorem \ref{thm:SignOfMulticonformalClass}:

\begin{enumerate}
\item If $[\![ g ]\!]$ contains a metric $\tilde{g}$ of positive scalar curvature, then it follows from \eqref{eq:KeqInequality} that
\begin{align*}
0 < R^{\tilde{g}}  \leq  \sum_i R^{g_i} \int_M f_i^{-2} d\mu_g.
\end{align*}
As $f_1, \ldots, f_l$ are positive functions there has to be at least one $i \in \lbrace 1, \ldots, l$ such that $\mu(M_i, [g_i ]) > 0$. On the other hand, if $\mu(M_i, [g_i]) > 0$ for some $i$, then an appropriate scaling of the single factors leads to a positive scalar curvature metric.

\item If $[\![ g ]\!]$ does not contain a metric of positive scalar curvature but a scalar flat metric $\tilde{g}$, then
\begin{align*}
0 = R^{\tilde{g}} \leq \sum_i R^{g_i} \int_M f_i^{-2} d\mu_g
\end{align*}
by \eqref{eq:KeqInequality}. Moreover, $\mu(M_i, [g_i]) \leq 0$ for all $1 \leq i \leq l$ as otherwise there would be a metric of positive scalar curvature in $[\![ g ]\!]$. Thus, it follows that the above inequality is satisfied if and only if $R^{g_i} = 0$, i.e.\  $\mu(M_i, [g_i]) = 0$ for all $1 \leq i \leq l$. As in this case the above inequality is in fact an equality the functions $f_1, \ldots, f_l$ have to be constant. In particular, the scalar flat metric $\tilde{g}$ is a product metric.
\end{enumerate}
Lastly, (3) is now an immediate consequence from (1) and (2). 
\end{proof}

As remarked in Sect.~\ref{sect:EinsteinHilbertRestrictions}, $\sigma(M)>0$ (resp.~$\sigma(M,  [\![ g ]\!])>0$) if and only if $M$ (resp.~$ [\![ g ]\!]$) carries a metric of positive scalar curvature. However, it is well known that the Kazdan--Warner trichotomy (cf.~\cite[Theorem 0.1]{MR2408269}) does not completely correspond to the sign of $\sigma(M)$. That is, if $M$ does not carry a positive scalar curvature metric but a scalar flat one, then $\sigma(M)=0$, but the converse does not hold in general. We note that a similar discrepancy holds for $\sigma(M,  [\![ g ]\!])$. That is, in Case (2) of Theorem \ref{thm:SignOfMulticonformalClass}, we have $\sigma(M,  [\![ g ]\!])=0$, while if $\mu(M_i, [g_i])\le 0$ for every $i$,  $\mu(M_i, [g_i])=0$ for some $i$, and $\mu(M_j, [g_i])<0$ for some $j$, then $\sigma(M,  [\![ g ]\!])=0$ but we are in Case (3). It is interesting to ask whether $\sigma(M,  [\![ g ]\!])<0$ holds if $\mu(M_i, [g_i])<0$ for every $i$. 

\begin{remark}\label{rmk:DimensionalAssumption}
Theorem \ref{thm:SignOfMulticonformalClass} has a technical assumption that no factor $M_i$ can be diffeomorphic to $S^1$. 
In the other extreme case where every factor $M_i$ is diffeomorphic to $S^1$, the enlargeability obstruction (Gromov--Lawson \cite{MR569070, MR720933}) or the stable minimal hypersurface obstruction (Schoen--Yau \cite{MR541332, MR535700, Schoen:2017aa}) show that $(M, g)=S^1(1)\times\cdots \times S^1(1)$ falls into case (2) of Theorem \ref{thm:SignOfMulticonformalClass}. It is interesting to ask whether our dimensional assumption $m_1, \dots, m_l\ge 2$ can be removed. The difficulty is that, as soon as there is an $S^1$-factor, one can conformally deform the flat metric $d\theta^2$ without changing the scalar curvature. 
\end{remark}

\section{The infimum of the Yamabe constants}\label{sect:InfYamabeConst}

In Sect.~\ref{sect:MulticonformalSign}, we were concerned with the supremum of the Yamabe constants within a multiconformal class. 
Theorem \ref{thm:InfYamabeConstant} below shows that its infimum is always $-\infty$. 

\begin{lemma}\label{lem:YamabeInf}
Let $(M^m, g)=(M_1^{m_1}, g_1)\times \dots\times (M_l^{m_l}, g_l)$ be a direct product of closed connected Riemannian manifolds, $l\ge 2$, and $m_1, \dots, m_l\ge 1$.  
If there exist $i\in\{1, \dots, l\}$ and $f_1, \dots, f_l: M\to\real_{>0}$ so that 
${\color{blue}\int_M\frac{R^g_i+\rho_i}{f_i^2}f_1^{m_1}\cdots f_l^{m_l}d\mu^g}<0$, 
then $\inf_{\tilde{g}\in [\![ g ]\!]}E(\tilde{g})=-\infty$.  
\end{lemma}

\begin{proof}
Define $\tilde{g}_{\varepsilon}=f_1^2g_1\oplus  \cdots \oplus  f_{i-1}^2g_{i-1} \oplus  \varepsilon^2f_i^2g_i \oplus  f_{i+1}^2g_{i+1} \oplus  \cdots \oplus  f_l^2g_l$. The scale invariant property of $\rho_i$ shows 
\begin{align*}
E(\tilde{g}_{\varepsilon})
=
\varepsilon^{2\left(\frac{m_i}{m}-1\right)}
	\frac{{\color{blue}\int_M\frac{R^g_i+\rho_i}{f_i^2}f_1^{m_1}\cdots f_l^{m_l}d\mu^g}}
		{\left(\int_Mf_1^{m_1}\cdots f_l^{m_l}d\mu^g\right)^{\frac{m-2}{m}}}
+\varepsilon^{2\frac{m_i}{m}}
	\sum_{j\neq i}^l\frac{{\color{blue}\int_M\frac{R^g_j+\rho_j}{f_j^2}f_1^{m_1}\cdots f_l^{m_l}d\mu^g}}
		{\left(\int_Mf_1^{m_1}\cdots f_l^{m_l}d\mu^g\right)^{\frac{m-2}{m}}}.  
\end{align*}
Since $m_i<m$, $\lim_{\varepsilon\to 0}E(\tilde{g}_{\varepsilon})=-\infty$.  
\end{proof}

\begin{theorem}\label{thm:InfYamabeConstant}
Let $(M^m, g=\langle \cdot, \cdot \rangle)=(M_1^{m_1}, g_1)\times \dots\times (M_l^{m_l}, g_l)$ 
be a direct product of closed connected Riemannian manifolds, 
where $l\ge 2$ and $m_1, \dots, m_l\ge 1$.  
If $m\ge 3$, then 
\begin{align}\label{eq:InfYamabeConstant}
\inf_{[\tilde{g}]\subset [\![ g ]\!]}\mu(M, [\tilde{g}])=\inf_{\tilde{g}\in [\![ g ]\!]}E(\tilde{g})
=-\infty.  
\end{align}
\end{theorem}

\begin{proof}
Suppose there exist $\varphi: M\to\real$ and $a_j:\real\to\real$ such that $\log f_j=a_j\circ \varphi$ for all $j\in\{1, \dots, l\}$.  
We may assume without loss of generality $l=2$, $m_1\ge 2$. For $i=1$, 
\begin{align*}
&\int_M\frac{\rho_i}{f_i^2}f_1^{m_1}\cdots f_l^{m_l}d\mu^g\\
&=(m_1-1)(m_1-2)\int_M(a_1'\circ\varphi)^2\lvert \grad_1 \varphi\rvert^2 f_1^{m_1-2}f_2^{m_2}d\mu^g\\
&\quad+m_2(m_2-1)\int_M(a_2'\circ\varphi)^2\lvert\grad_1\varphi\rvert^2f_1^{m_1-2}f_2^{m_2}d\mu^g\\
&\quad+2(m_1-1)m_2\int_M(a_1'\circ\varphi)(a_2'\circ\varphi)\lvert\grad_1\varphi\rvert^2 f_1^{m_1-2}f_2^{m_2}d\mu^g.  
\end{align*}
For real numbers $\alpha, \beta>0$ to be specified later, set $a_1(\theta)=\sin(\sqrt{\alpha}\theta)$, $a_2(\theta)=-\beta\sin(\sqrt{\alpha}\theta)$.  
Compute
\begin{align*}
&(m_1-1)(m_1-2)(a_1'(\theta))^2
+m_2(m_2-1)(a_2'(\theta))^2
+2(m_1-1)m_2(a_1'(\theta))(a_2'(\theta))\\
&=\alpha\cos^2\theta\left((m_1-1)(m_1-2)-2(m_1-1)m_2\beta+m_2(m_2-1)\beta^2\right)\\
&=:-\alpha\gamma\cos^2\theta
\end{align*}
and observe that $\gamma>0$ for a good choice of $\beta$, provided that $m_1\ge 2$ and $m_2\ge 1$.  
More precisely, if $m_2=1$, then take $\beta$ large enough; 
if $m_2\ge 2$, then take $\beta>0$ slightly smaller than the larger root of the corresponding quadratic form.  
Therefore, for $i=1$, 
\begin{align*}
&\int_M\frac{\rho_i}{f_i^2}f_1^{m_1}\cdots f_l^{m_l}d\mu^g\\
&=\int_MR_1e^{(m_1-2-\beta m_2)\sin(\sqrt{\alpha}\varphi)}d\mu^g
	-\alpha\gamma\int_M(\cos^2\varphi)\lvert\grad_1\varphi\rvert^2e^{(m_1-2-\beta m_2)\sin(\sqrt{\alpha}\varphi)}d\mu^g\\
&\le e^{\lvert m_1-2-\beta m_2\rvert}\max_{M_1}\lvert R_1\rvert\int_Md\mu^g
	-\alpha\gamma e^{-\lvert m_1-2-\beta m_2\rvert}\int_M (\cos^2\varphi)\lvert\grad_1\varphi\rvert^2d\mu^g.  
\end{align*}
For the fixed $\beta>0$, we can take $\alpha>0$ so large that the right hand side is negative. 
Hence \eqref{eq:InfYamabeConstant} holds by Lemma \ref{lem:YamabeInf}. 
\end{proof}

Note that Theorem \ref{eq:InfYamabeConstant} implies Theorem \ref{thm:main} (2), 
since every conformal class $[g]$ on a closed connected manifold $M$ has a unit-volume metric of scalar curvature constantly equal to the Yamabe constant $\mu(M, [g])$. 

\section{Multiconformal metrics of permutation type}\label{sect:WarpedProduct}
In Sect.~\ref{sect:InfYamabeConst}, we saw that the scalar curvature of a multiconformal metric $\tilde{g}\in [\![g]\!]$ can be negative everywhere even in the case (1) of Theorem \ref{thm:SignOfMulticonformalClass}. 
In this section, we show that such negative scalar curvature metrics cannot be of permutation type in the sense of Definition \ref{def:PermutationType}. 

Let $(M, g)=(M_1, g_1)\times\cdots\times (M_l, g_l)$ be a direct product of Riemannian manifolds. 
For functions $f_1, \dots, f_l\in C^{\infty}_+(M)$, we may associate the $(l\times l)$-matrix 
\begin{align}\label{eq:JacobianMatrix}
\begin{pmatrix}
d_1f_1&\cdots &d_1f_l\\
\vdots&&\vdots\\
d_lf_1&\cdots &d_lf_l
\end{pmatrix}
\end{align}
of $1$-forms on $M$. 
In view of this matrix, we impose the following conditions on the multiconformal factors. 

\begin{definition}\label{def:PermutationType}
Let $\tilde{g}=f_1^2g_1\oplus\cdots\oplus f_l^2g_l$. 
We say $\tilde{g}$ has off-diagonal type if $f_i$ is constant along $M_i$ for every $i$. 
For a map $\sigma: \{1, \dots, l\}\to \{1, \dots, l\}$, we say $\tilde{g}$ has type $\sigma$ if $f_i$ is constant along $M_1\times\cdots\times M_{\sigma(i)-1}\times M_{\sigma(i)+1}\times\cdots\times M_l$ for every $i$. 
If $\tilde{g}$ has type $\sigma$ for some bijection $\sigma$, then $\tilde{g}$ is said to have permutation type. 
\end{definition}

\noindent We remark that a multiconformal metric $\tilde{g}=f_1^2g_1\oplus\cdots\oplus f_l^2g_l$ has off-diagonal and permutation type, respectively, if and only if \eqref{eq:JacobianMatrix} has zero diagonal entries and is a generalized permutation matrix. 

\begin{remark}\label{rmk:terminology}
Related terminology is the following. 
The notion of \emph{warped products} in the sense of Bishop--O'Neill (cf.~\cite[\S 7]{MR0251664}, \cite[\S 7]{MR719023}) 
has been generalized to various situations.  
We remark that \emph{doubly warped products} can have two different meanings; 
some authors\footnote{
	e.g.~Allison \cite[Definition 2.2]{MR1119654}, 
	Yang \cite[p.~203]{MR1651555}, 
	\"Unal \cite[Definition 2.1]{MR1868561}, 
	Brozos-V\'azquez--Garc\'\i a-R\'\i o--V\'azquez-Lorenzo \cite[Remark 5]{MR2121707}, 
	Olteanu \cite[Definition 1]{MR3525858}.  
} deal only with two factors while others\footnote{
	e.g.~Zucker \cite[p.~215]{MR684171}, 
	Gromov--Lawson \cite[p.~188]{MR720933}, 
	Ivey \cite{MR1207538}, 
	Petersen \cite[Chapter 1, \S 4]{MR2243772}, 
	Walsh \cite[p.~6]{MR2789750}.  
} need three factors to define them.  
The term \emph{multiply warped products} seems to be unambiguous\footnote{
	cf.~Br\"uning \cite[p.~303]{MR1162191}, 
	\"Unal \cite[Definition 2.1]{MR1762779}, 
	Dobarro--\"Unal \cite[Definition 2.1]{MR2157416}, 
	U\u{g}uz\linebreak[0]--\linebreak[0]Bilge \cite[\S 2.2]{MR2782484}, 
	Chen \cite[p.~13]{MR3525850}.  
}, but it conflicts with the first meaning of doubly warped products.  

\emph{Twisted products} in the sense of Chen \cite[p.~66]{MR627323}, 
also called \emph{umbilic products} in earlier work of Bishop \cite[p.~27]{MR0334078}, 
are defined on direct product manifolds, which are topologically not twisted.  
Note that Bishop--O'Neill \cite[p.~29]{MR0251664} used the term \emph{warped bundles} 
for the generalization of warped products to (possibly topologically twisted) bundles.  
These notions are generalized, 
depending on the authors' preferences, 
to 
\emph{umbilic products}\footnote{
	cf.~Gauchman \cite[Definition 1]{MR640977}.  
}, \emph{twisted products}\footnote{
	cf.~Koike \cite[p.~3]{MR1212293}, 
	Meumertzheim--Reckziegel--Schaaf \cite[Definition 2]{MR1726218}.  
}, and \emph{doubly} or \emph{multiply twisted products}\footnote{
	cf.~Ponge--Reckziegel \cite[p.~15]{MR1245571}, 
	Rovenskii \cite[Definition 2.6] {MR1770408}, 
	Fern\'andez-L\'opez--Garc\'\i a-R\'\i o--Kupeli--\"Unal \cite[p.~214]{MR1865565}, 
	Uddin \cite[p.~35]{MR2639329}, 
	Wang \cite[p.~1]{MR3206817}.  
} etc. 
\end{remark}

The following implies Theorem \ref{thm:sub2} (3). 
\begin{theorem}\label{thm:NotOfAnyWarpedProductType}
Let $(M, g)=(M_1, g_1)\times\cdots\times (M_l, g_l)$ be a direct product of closed Riemannian manifolds, 
and assume $R^g\ge 0$. 
If a multiconformal metric $\tilde{g}=f_1^2g_1\oplus\cdots\oplus f_l^2g_l$ has permutation type and $R^{\tilde{g}}\le 0$, 
then $f_1, \dots, f_l$ are constant; 
in particular, $R^{\tilde{g}}\equiv R^g\equiv 0$. 
\end{theorem}

\begin{proof}
Let $\tilde{g}=f_1^2g_1\oplus\cdots\oplus f_l^2g_l$ has type $\sigma$ for some permutation $\sigma$. 
We may assume $\sigma(i)\neq i$ for every $i$. 
Indeed, if $\sigma(j)=j$ for some $j$, 
then $(M, \tilde{g})$ is isomorphic to 
\begin{align*}
(M_j, f_j^2g_j)\times \left(\prod_{i\neq j}M_i, \bigoplus_{i\neq j}f_i^2g_i\right)
\end{align*}
The integral formula \eqref{eq:IntegralRhoi} then simplifies to 
\begin{align*}
{\color{blue}\int_M\frac{\rho_i}{f_i^2}f_1^{q_1}\cdots f_l^{q_l}d\mu^g}
&=(m_i-1)(2q_i-m_i-2)\int_M\frac{{\color{blue}\lvert \grad^g_i f_i\rvert^2}}{f_i^4} f_1^{q_1}\cdots f_l^{q_l}d\mu^g\\
&\quad+\sum_{j\neq i}m_j(2q_j-m_j-1)\int_M\frac{{\color{olive}\lvert\grad^g_if_j\rvert^2}}{f_i^2f_j^2} f_1^{q_1}\cdots f_l^{q_l}d\mu^g, 
\end{align*}
so we obtain
\begin{align}\label{eq:ProofWarpedAux}\begin{split}
&\int_M\left(R^{\tilde{g}}-\sum_{i=1}^l\frac{R^g_i}{f_i^2}\right)f_1^{q_1}\cdots f_l^{q_l}d\mu_g
=\sum_{i=1}^l\int_M\frac{\rho_i}{f_i^2}f_1^{q_1}\cdots f_l^{q_l}d\mu^g\\
&=\sum_{i=1}^l(m_i-1)(2q_i-m_i-2)\int_M\frac{{\color{blue}\lvert \grad^g_i f_i\rvert^2}}{f_i^4} f_1^{q_1}\cdots f_l^{q_l}d\mu^g\\
&\quad+\sum_{i\neq j}m_j(2q_j-m_j-1)\int_M\frac{{\color{olive}\lvert\grad^g_if_j\rvert^2}}{f_i^2f_j^2} f_1^{q_1}\cdots f_l^{q_l}d\mu^g. 
\end{split}\end{align}

For $q_1, \dots, q_l$ chosen largely enough, the right hand side of \eqref{eq:ProofWarpedAux} is nonnegative. 
Hence $R^{\tilde{g}}-\sum_{i=1}^lR^g_i/f_i^2$ cannot be nonnegative and negative at one point, 
since otherwise the left hand side of \eqref{eq:ProofWarpedAux} would become negative. 
\end{proof}

Yang \cite[Theorem 1]{MR1651555} observed that if $(M_1\times M_2, \tilde{g}=f_1^2g_1\oplus  f_2^2g_2)$ is of warped product type and has constant scalar curvature and if $f_i$ is nonconstant for $i=1, 2$, then $R^{\tilde{g}}$ must be indeed zero. 
He then asked whether there exist such scalar flat metrics of nontrivial warped product type. 
Before providing an answer to his question, we generalize Yang's Theorem slightly as follows. 

\begin{theorem}[Kwang-Wu Yang]
Assume that there exists a permutation $\sigma$ of $l$ letters so that $\sigma(i)\neq i$ and $f_i$ is constant along $M_1\times\cdots\times M_{\sigma(i)-1}\times M_{\sigma(i)+1}\times\cdots\times M_l$ for every $i\in\{1, \dots, l\}$. 
If $\tilde{g}=f_1^2g_1\oplus\cdots\oplus f_l^2g_l$ has constant scalar curvature, and if $f_i$ is nonconstant for every $i\in\{1, \dots, l\}$, then $R^{\tilde{g}}=0$. 
\end{theorem}

\begin{proof}
The formula \eqref{eq:ScalarCurvatureDifference} simplifies to 
\begin{align}\label{eq:YangTheoremAux1}\begin{split}
R^{\tilde{g}}
&=\sum_{i}\frac{R_i^g}{f_i^2}
	-2\sum_{i\neq j}m_j\frac{\Delta^g_if_j}{f_i^2f_j} 
	-\sum_{i\neq j}m_j(m_j-1)\frac{\lvert\grad^g_if_j\rvert^2}{f_i^2f_j^2}\\
&=\sum_{i=1}^l\frac{1}{f_i^2}\left( R^g_i-2m_{\sigma^{-1}(i)}\frac{\Delta^g_if_{\sigma^{-1}(i)}}{f_{\sigma^{-1}(i)}}-m_{\sigma^{-1}(i)}(m_{\sigma^{-1}(i)}-1)\frac{\lvert\grad^g_if_{\sigma^{-1}(i)}\rvert^2}{f_{\sigma^{-1}(i)}^2}\right). 
\end{split}\end{align}
We fix $i\in\{1, \dots, l\}$ and take a vector field $X$ tangent to $E_{\sigma(i)}$. 
Differentiation of \eqref{eq:YangTheoremAux1} with respect to $X$ yields 
\begin{align*}
0&=X(R^{\tilde{g}})\\
&=
X\left(\frac{1}{f_{i}^2}\right)\cdot\left( R^g_i-2m_{\sigma^{-1}(i)}\frac{\Delta^g_if_{\sigma^{-1}(i)}}{f_{\sigma^{-1}(i)}}-m_{\sigma^{-1}(i)}(m_{\sigma^{-1}(i)}-1)\frac{\lvert\grad^g_if_{\sigma^{-1}(i)}\rvert^2}{f_{\sigma^{-1}(i)}^2}\right)\\
&\quad+\frac{1}{f_{\sigma(i)}^2}X\left(R_{\sigma(i)}^g-2m_{i}\frac{\Delta^g_{\sigma(i)}f_{i}}{f_{i}}-m_{i}(m_{i}-1)\frac{\lvert\grad^g_{\sigma(i)}f_{i}\rvert^2}{f_{i}^2} \right). 
\end{align*}
In other words, 
\begin{align}\label{eq:YangTheoremAux2}\begin{split}
&X\left(\frac{1}{f_{i}^2}\right)\cdot {\color{blue}f_{\sigma(i)}^2 \left( R^g_i-2m_{\sigma^{-1}(i)}\frac{\Delta^g_if_{\sigma^{-1}(i)}}{f_{\sigma^{-1}(i)}}-m_{\sigma^{-1}(i)}(m_{\sigma^{-1}(i)}-1)\frac{\lvert\grad^g_if_{\sigma^{-1}(i)}\rvert^2}{f_{\sigma^{-1}(i)}^2}\right)}\\
&=- X\left(R_{\sigma(i)}^g-2m_{i}\frac{\Delta^g_{\sigma(i)}f_{i}}{f_{i}}-m_{i}(m_{i}-1)\frac{\lvert\grad^g_{\sigma(i)}f_{i}\rvert^2}{f_{i}^2} \right). 
\end{split}\end{align}
We claim that, for every $i\in\{1, \dots, l\}$, 
\begin{align*}
{\color{blue}f_{\sigma(i)}^2\left( R^g_i-2m_{\sigma^{-1}(i)}\frac{\Delta^g_if_{\sigma^{-1}(i)}}{f_{\sigma^{-1}(i)}}-m_{\sigma^{-1}(i)}(m_{\sigma^{-1}(i)}-1)\frac{\lvert\grad^g_if_{\sigma^{-1}(i)}\rvert^2}{f_{\sigma^{-1}(i)}^2}\right)}
\equiv c_i
\end{align*}
for some $c_i\in\real$. 
Indeed, the right hand side of \eqref{eq:YangTheoremAux2} is constant along $M_1\times\cdots\times M_{\sigma(i)-1}\times M_{\sigma(i)+1}\times\cdots\times M_l$, while its left hand side is the multiple of $X\left(\frac{1}{f_{i}^2}\right)$, which is constant along the same space and is not identically equal to zero for a good choice of $X$, and the rest which is constant along $M_{\sigma(i)}$. 
Going back to \eqref{eq:YangTheoremAux1}, we obtain
\begin{align*}
R^{\tilde{g}}=\sum_{i=1}^l\frac{c_i}{f_i^2f_{\sigma(i)}^2}. 
\end{align*}
Differentiation with respect to a vector field $X$ tangent to $M_{\sigma(i)}$ then yields 
\begin{align*}
0=X(R^{\tilde{g}})=X\left(\frac{1}{f_i^2}\right)\left(\frac{c_i}{f_{\sigma(i)}^2}+\frac{c_{\sigma^{-1}(i)}}{f_{\sigma^{-1}(i)}^2}\right). 
\end{align*}
A similar separation of variables then shows 
$\frac{c_i}{f_{\sigma(i)}^2}+\frac{c_{\sigma^{-1}(i)}}{f_{\sigma^{-1}(i)}^2}=0$, whence $c_i=0$ for every $i$. 
Therefore, $R^{\tilde{g}}=0$. 
\end{proof}

Following Yang, we ask whether there exist such scalar flat metrics. 
Theorems \ref{thm:SignOfMulticonformalClass}, \ref{thm:NotOfAnyWarpedProductType} imply the following partial answer to his question. 
If either $\mu(M_i, [g_i])\le 0$ for $i=1, 2$ or $R^{g_i}\ge 0$ for $i=1$ and $2$, then such scalar flat metrics of nontrivial warped product type do not exist. 

\section*{Acknowledgements}
We thank Professors Bernd Ammann and Naoyuki Koike for their comments on the previous version of this article. 
N.~Otoba is supported by the DFG (Deutsche Forschungsgemeinschaft), SFB 1085 Higher Invariants. S.\ Roos was supported by the Hausdorff Center for Mathematics in Bonn.

%
\bibliographystyle{amsplain}
\bibliography{/Users/Nobu/GoogleDrive/References_otoba}
\end{document}